\UseRawInputEncoding
\documentclass[a4paper,10pt,reqno]{amsart}
\usepackage[utf8]{inputenc}  
\usepackage[T1]{fontenc}     
\usepackage{textcomp}   
\usepackage{amssymb}
\usepackage{graphicx}
\usepackage[margin=1in]{geometry}  
\usepackage[dvipsnames]{xcolor}
\definecolor{Chocolat}{rgb}{0, 0.15, 0.05}
\definecolor{BleuTresFonce}{rgb}{0, 0.15, 0.05}
\definecolor{SoftTeal}{HTML}{6FAFB3}
\definecolor{MutedLavender}{HTML}{B8A9C9}
\definecolor{WarmSand}{HTML}{E6D3A3}
\definecolor{DustyRose}{HTML}{7E9F8A}
\definecolor{SageGreen}{HTML}{D8A1A9}
\definecolor{SlateBlue}{HTML}{B8445A}
\definecolor{MutedIndigo}{HTML}{847FC6}
\definecolor{DustyRed}{HTML}{5F8FC9}
\definecolor{EggShell}{RGB}{244, 241, 222}
\definecolor{Graphite}{HTML}{7A6F82}
\definecolor{DarkGraphite}{HTML}{4A4E57}
\definecolor{softblack}{HTML}{241815}
\definecolor{espresso}{HTML}{3A2A24}
\usepackage[utf8]{inputenc}
\usepackage[colorlinks,hyperindex]{hyperref}
\hypersetup{linkcolor=espresso,urlcolor=espresso,citecolor=espresso}

\theoremstyle{plain}
\newtheorem{theorem}{Theorem}

\newtheorem{lemma}{Lemma}

\newtheorem*{claim*}{Claim}
\newtheorem*{thm*}{Theorem}

\theoremstyle{definition}
\newtheorem{definition}{Definition}
\newtheorem{convention}{Convention}

\newtheorem{remark}{Remark}

\newcommand{\cal}[1]{\mathcal{#1}}              

 \usepackage{multicol}

\usepackage{float}
\usepackage{amssymb}
\usepackage{stmaryrd}
\usepackage{mathalfa}
\usepackage{euscript}
\usepackage{eufrak}
\usepackage{extarrows}
\usepackage[normalem]{ulem}
\usepackage{amsmath}
\usepackage{amsthm}
 \let\oldtocsection=\tocsection
\let\oldtocsubsection=\tocsubsection

\renewcommand{\tocsection}[2]{\hspace{0em}\vspace{0.1em}\rule{0pt}{14pt}\oldtocsection{#1}{#2}\bf}
\renewcommand{\tocsubsection}[2]{\hspace{2em}\oldtocsubsection{#1}{#2}}

\numberwithin{equation}{section}

\allowdisplaybreaks
\usepackage{lscape}
\usepackage{tabularx}

\usepackage{framed}
\usepackage[title]{appendix}
\usepackage{tikz}
\usetikzlibrary{arrows.meta}
\usetikzlibrary{decorations.markings}
\usetikzlibrary{shapes.geometric,positioning}
\usepackage{booktabs}
\usepackage{euscript}
\usepackage[normalem]{ulem}
\usepackage{array}
\usepackage{collectbox}
\usetikzlibrary{arrows,matrix}

\usepackage{footnote}
\usepackage{makecell}
\pgfdeclarelayer{edgelayer}
\pgfdeclarelayer{nodelayer}
\pgfsetlayers{edgelayer,nodelayer,main}
\usepackage{array}
\usepackage{amssymb}
\usetikzlibrary{decorations.pathmorphing, patterns,shapes}
 
\usepackage{tikz-cd}
\usetikzlibrary{cd}

\usepackage{multirow}

 \usetikzlibrary{calc}
\usepackage{accents}
\definecolor{bluen}{RGB}{0,100,200}
 
\makeatletter

\pgfarrowsdeclare{open cap}{open cap}
{\pgfarrowsleftextend{+0pt}\pgfarrowsrightextend{+0.5\pgflinewidth}}
{
  \pgfmathsetlength{\pgfutil@tempdimb}{.5*\pgflinewidth-.5*\pgfinnerlinewidth}%
  \pgfsetlinewidth{\pgfutil@tempdimb}
  \pgfsetbuttcap
  \pgfsetdash{}{0pt}
  \pgfmathsetlength{\pgfutil@tempdima}{.5*\pgfutil@tempdimb+.5*\pgfinnerlinewidth}%
  \pgfpathmoveto{\pgfqpoint{0pt}{\pgfutil@tempdima}}
  \pgfpatharc{270}{90}{-\pgfutil@tempdima}
  \pgfusepathqstroke
}
\makeatother

\tikzset{snake it/.style={decorate, decoration=snake}}
\usetikzlibrary{decorations.pathreplacing}
\newcommand\mydots{\hbox to 1em{.\hss.\hss.}}
\subjclass[2020]{Primary 18F20; Secondary 18M30}


\thanks
{The author would like to thank François Lamarche who generously shared the idea and motivation to treat the B\' enabou-Roubaud result  in a studious and novel way. The author would also like to thank  Pierre-Louis Curien  for very useful conversations concerning this topic.  }

\title[B\' enabou-Roubaud theorem via string diagrams]{The B\' enabou-Roubaud  theorem via string diagrams}

\author{Jovana Obradovi\' c}
\address{Mathematical Institute, Serbian Academy of Sciences and Arts, Belgrade, Serbia}
\email{jovana@mi.sanu.ac.rs}

\date{\today}
 
\begin{document}

\begin{abstract}
We give  a complete proof of  the B\' enabou-Roubaud monadic descent theorem using the graphical calculus  of string diagrams.    
Our proof links  the monadic and Grothendieck's original viewpoint on descent  via an internal-category-based characterization of the
 category of descent data, equivalent to the one of Janelidze and Tholen.   This article provides a formal and self-contained account of the author’s unpublished 2016 note. 
\end{abstract}

\maketitle

\thispagestyle{empty}

\tableofcontents

\section{Introduction}
\noindent  Given a   fibration  $p:{\cal E}\rightarrow {\cal C}$ and a morphism $f:B\rightarrow A$ in the base category ${\cal C}$,   Grothendieck's   descent theory \cite{gro} studies the problem of reconstructing the fiber ${\cal E}_A$ from the fiber ${\cal E}_B$ with respect to the induced  pullback
functor $f^{*}:{\cal E}_A\rightarrow {\cal E}_B$. Here, the loss of information caused by $f^{\ast}$ is compensated by additional structure on the objects of ${\cal E}_B$, known as {\em descent data}. In this way,  one   defines the   category of descent data ${\mathsf{Desc}}_{p}(f)$, together with a comparison functor $\Phi^f_p:{\cal E}_A\rightarrow {\mathsf{Desc}}_{p}(f)$,  which  induces the  descent factorization
\begin{equation}
\raisebox{-0.8cm}{\begin{tikzpicture}[scale=1.5]
\node (A) at (-0.4,0) {${\cal E}_A$};
\node (B) [inner sep=0.3mm] at (0.75,0.5) {${\mathsf{Desc}}_{p}(f)$};
\node (C) at (1.95,0) {${\cal E}_B$};
\draw[->] (A) edge node[yshift=-0.25cm]{\small  $f^{\ast}$} (C);
\draw[->] (A) edge node[left,yshift=0.2cm,xshift=0.1cm]{\small $\Phi^f_p$} (B);
\draw[->] (B) edge node[right,yshift=0.2cm,xshift=-0.2cm]{\small  $U_f$} (C);
\end{tikzpicture} }
\label{descfact}
\end{equation}
in which $U_f$ is the forgetful functor.   One then asks whether   $\Phi^f_p$ is an equivalence of categories.  Morphisms $f$ for which this holds are called {\em effective descent morphisms}.

\smallskip

In their seminal 1970 paper \cite{br},  B\' enabou and Roubaud  observed that,  for   a bifibration $p:{\cal E}\rightarrow {\cal C}$  satisfying the Beck-Chevalley condition,  the descent factorization \eqref{descfact} coincides with the Eilenberg-Moore factorization  
\begin{equation}
\raisebox{-0.8cm}{\begin{tikzpicture}[scale=1.5]
\node (A) at (-0.4,0) {${\cal E}_A$};
\node (B) [inner sep=0.3mm] at (0.75,0.5) {${\mathsf{EM}}_{p}(T_f)$};
\node (C) at (1.95,0) {${\cal E}_B$};
\draw[->] (A) edge node[yshift=-0.25cm]{\small  $f^{\ast}$} (C);
\draw[->] (A) edge node[left,yshift=0.2cm,xshift=0.1cm]{\small $\Phi^{T_f}_p$} (B);
\draw[->] (B) edge node[right,yshift=0.2cm,xshift=-0.2cm]{\small  $U_{T_f}$} (C);
\end{tikzpicture} }
\label{emfact}
\end{equation}
associated to the adjunction $f^{\ast}\dashv f_!:{\cal E}_B\rightarrow {\cal E}_A$,
 up to an equivalence of categories ${\mathsf{EM}}_{p}(f)\simeq {\mathsf{Desc}}_{p}(f)$. As a consequence, they obtained a characterization of effective descent morphisms by means of monadicity:  $f:B\rightarrow A$ is an effective descent morphism precisely when the functor $f^{*}:{\cal E}_A\rightarrow {\cal E}_B$
is monadic, i.e., when the canonical comparison  functor $\Phi^{T_f}_p$ in \eqref{emfact} is an equivalence of categories.  However, their article did not attempt to provide a proof  of the equivalence ${\mathsf{Desc}}_{p}(f)\simeq {\mathsf{EM}}_{p}(f)$.

\smallskip

In   \cite{th},  Janelidze and Tholen  introduced a third   description  of the category of descent data,   identifiying it with the category $\mathsf{Act}_p^{{\mathsf{Eq}}(f)}$ of actions of the internal category ${\mathsf{Eq}}(f)$, induced by the kernel pair of $f$, on $p$, while treating the B\' enabou-Roubaud theorem as folklore. 

\smallskip

Over the years, a number of other authors have also returned to the topic, but a complete written proof of  the B\' enabou-Roubaud theorem has remained elusive.  In   more modern developments, such as the work \cite{FLN} of  Lucatelli Nunes, the theorem  was established only as a  consequence of    more general results  appealing to higher-level machinery. Only very recently, in his 2024 article \cite{BK},   Kahn filled this gap by presenting a detailed, rigorous proof of the  B\' enabou-Roubaud theorem in the original 1-categorical context. Nevertheless,  in order to simplify the calculations, he  rectified the pseudofunctor $f\mapsto f^{\ast}$ to an honest functor, which  eliminated the need to manage pseudofunctorial coherence isomorphisms. Even with this ``bypass'', the proof contains long series of explicit calculations.
 
\smallskip

The aim of this paper is to supply a  different kind of  proof of the  B\' enabou-Roubaud theorem: a graphical proof  based on the string diagram  calculus built on the {\em    bifibrational signature}, that is, the 2-polygraph (in the sense of Burroni, see \cite{bur}) with relations encoding the   data of  pseudofunctoriality of $f\mapsto f^{\ast}$ and
the data of the adjunction $f^{\ast}\dashv f_!:{\cal E}_B\rightarrow {\cal E}_A$, including the (invertible) Beck-Chevalley transformation. By the soundness and completeness of the graphical calculus of string diagrams for 2-categories, established by Joyal and Street in \cite{JoyalStreet1991}, the 2-cells of the  2-category ${\mathsf{BiFib}}_{\cal C}$ built on the same signature may be identified with isotopy classes of those string diagrams,  and their equality 
 is established by finite sequences of diagrammatic rewritings. The objects of the three descent categories $${\mathsf{EM}}_{p}(f), \quad {\mathsf{Desc}}_{p}(f) \quad \mbox{ and } \quad \mathsf{Act}_p^{{\mathsf{Eq}}(f)}$$ can then be expressed  by extending the string diagram  calculus of ${\mathsf{BiFib}}_{\cal C}$  in three ways, each time with a single new generating 2-cell, encoding a particular way of characterizing a descent datum, subject to  appropriate relations, and their equivalence is established in the same diagrammatic fashion, which  provides conceptual clarity of the result, making it
 effectively a matter of geometry rather than bookkeeping.

\smallskip

Although Kahn's work provides a complete algebraic proof of the B\' enabou-Roubaud theorem, we believe that the string-diagrammatic proof developed here is of independent value. First, it  takes into account the pseudofunctorial structure induced by a bifibration explicitly: all the pseudofunctorial coherence isomorphisms are treated as generating 2-cells and their coherence is proven internally in ${\mathsf{BiFib}}_{\cal C}$. Second, the diagrammatic formulation yields a proof  that is arguably simpler and more transparent: long chains  of componentwise algebraic calculations become, without loss of rigor, short sequences of local diagram rewritings, since the issues of associativity, functoriality and naturality are ``absorbed by the syntax''. Third, we provide an alternative description of the category $\mathsf{Act}_p^{{\mathsf{Eq}}(f)}$, and we explicitly link the original and monadic description of descent data with  this new characterization of descent data. And, finally, we hope that our proof can be useful in tracing descent in other contexts, such as categorical semantics of dependent type theory, in which  bifibrations satisfying Beck-Chevalley condition play a prominent role (cf. \cite{BJ} for a general reference on the subject, and \cite{pl0} for a more specialized analysis of the pseudofunctorial character of  pullback, which in fact uses string diagrams in order to prove certain properties of interpretations of dependent types).

\smallskip

 The article is split into three parts. Section \ref{prel}  is a recollection on bifibrations and string diagrams. Section \ref{syncat} introduces the string-diagrammatic category ${\mathsf{BiFib}}_{\cal C}$ and proves pseudofunctorial coherence  as a syntactic result, using simple a simple rewriting technique. Section \ref{thethm}  is  devoted to the B\' enabou-Roubaud theorem: we introduce  the three categories of descent data  and we provide a graphical proof of their equivalence.  

\smallskip

\paragraph{\bf Conventions.}   This paper  involves compositions of various kinds. In the 2-category ${\mathsf{Cat}}$ of categories, functors and natural transformations, we omit the symbol ``$\circ$'' when denoting the composition of functors; the symbol ``$\circ$'' is used explicitly for vertical composition of natural transformations and for composition of ``ordinary'' arrows inside a particular category. Horizontal composition of natural transformations (including whiskering)  will be denoted by ``$\cdot$''. We shall take the same convention for the syntactic  2-category  ${\mathsf{BiFib}}_{\cal C}$  and its extensions.  

\section{Preliminaries}\label{prel}
\noindent To make this article self-contained, in this section
  we review two basic concepts of the theory of 2-dimensional categories: {\it bifibrations} and their  pseudofunctorial structure, and {\em string diagrams}. The book \cite{NY} serves as our general reference  to 2-categories and bicategories,  containing in particular   more detailed treatment of both of these subjects. As more specialized  references, we recomend \cite{LR} for fibrations, and   \cite{dm} and \cite{D} for string diagrams. 
\subsection{The pseudofunctorial structure of a bifibration} 
A functor $p:{\cal E}\rightarrow {\cal C}$ is called a {\em Grothendieck fibration} or a {\em fibration}  if, for each morphism $f:B\rightarrow A$   in ${\cal C}$ and each object  $X\in {\it Ob}({\cal E}_A)$, where ${\cal E}_A$ denotes the fiber of $p$ over $A$, there exists a morphism  $\hat{f}:Y\rightarrow X$  in ${\cal E}$ such that:
\begin{itemize}
\item  $p(\hat{f})=f$, and 
\item $\hat{f}$ is {\em cartesian}: for every object $Z\in {\it Ob}({\cal E})$,  every arrow $h:Z\rightarrow X$ in ${\cal E}$ and every arrow  $g:p(Z)\rightarrow B$ in ${\cal C}$ such that $f\circ g=p(h)$, there exists a unique arrow $k:Z\rightarrow Y$ in ${\cal E}$ such that $\hat{f}\circ k=h$.
\end{itemize} The  morphism $\hat{f}:Y\rightarrow X$ is called a  {\em cartesian lift of $f$}.  Dually, a functor $p:{\cal E}\rightarrow {\cal C}$  is  an {\em opfibration} if the opposite functor $p^{\it op}:{\cal E}^{\it op}\rightarrow {\cal C}^{\it op}$ is a fibration. A  {\em bifibration} is a functor that is both a fibration and an opfibration.

\smallskip

If $p:{\cal E}\rightarrow {\cal C}$ is  a fibration, then, assuming the  axiom of choice,  one can pick a {\em cleavage over} ${\cal C}$ with respect to $p$, i.e., one can fix, for each  morphism $f:B\rightarrow A$   in ${\cal C}$ and each object $X\in {\it Ob}({\cal E}_A)$, a cartesian lift $\hat{f}_X: f^{\ast}X \rightarrow X$ of $f$ at $X$. By the universal property of cartesian lifts,  this  choice  extends functorially: each object $X\in {\it Ob}({\cal E}_A)$  is mapped to the domain $f^{\ast}X$ of $\hat{f}_X$, and each morphism $\gamma:X'\rightarrow X$ is mapped to the unique arrow $f^{\ast}\gamma:f^{\ast}X\rightarrow f^{\ast}X'$ making the diagram 
\begin{center}
\begin{tikzpicture}[scale=1.5]
\node (A) at (0,1) {$f^{\ast}X$};
\node (B) at (1,1) {$X$};
\node (C) at (0,0) {$f^{\ast}X'$};
\node (D) at (1,0) {$X'$};
\draw[->,dashed] (A) edge node[left]{$f^{\ast}\gamma$} (C);
\draw[->] (B) edge node[right]{$\gamma$} (D);
\draw[->] (A) -- (B);
\draw[->] (C)--(D);
\end{tikzpicture}
\end{center}
commute. The resulting functor $f^{\ast}:{\cal E}_A\rightarrow {\cal E}_B$ is called the pullback functor (a.k.a. reindexing or base-change functor).
The assignment  
$$\begin{array}{rcl}
A &\longmapsto & {\cal E}_A\\
f:B\rightarrow A &\longmapsto &f^{\ast}:{\cal E}_A\rightarrow {\cal E}_B
\end{array}$$
defines a  contravariant   pseudofunctor $(-)^{\ast} :{\cal C}\rightarrow \mathsf{Cat}$, since   identities and composition are preserved only up to canonical isomorphisms.  In fact,   fibrations (with a chosen cleavage) correspond precisely to such pseudofunctors, in a strict 2-equivalence of the corresponding 2-categories (see \cite[Theorem 10.6.16]{NY}). For an object $B\in{\cal C}$ and composable arrows $g:C\rightarrow B$ and $f:B\rightarrow A$  in ${\cal C}$,  the natural isomorphisms 
\begin{equation}\iota_B: {\text{Id}}_{{\cal E}_B}\Rightarrow ({\text{id}}_B)^{\ast}\quad\mbox{ and }\quad  \kappa_{f,g}:g^{\ast} f^{\ast}\Rightarrow (f\circ g)^{\ast}\label{canonical_isos_pseudo}\end{equation} witnessing the pseudofunctoriality of $(-)^{\ast}$ arise by the universal property of cartesian lifts as follows:  for $X\in{\it Ob}({\cal E}_{B})$, the component $\iota_{B,X}$ of  $\iota_B$  arises from the  unique factorization $$X\xrightarrow{\iota_{B,X}}(\mathrm{id}_B)^{\ast}X\xrightarrow{(\widehat{\mathrm{id}_B})_X} X$$ of ${\mathrm{id}}_X$,
and, for $Y\in{\it Ob}({\cal E}_A)$, the component $\kappa_{f,g,Y}$ of $\kappa_{f,g}$ arises from the unique factorization $$g^{\ast}f^{\ast}Y\xrightarrow{\kappa_{f,g,Y}}(f\circ g)^{\ast}Y\xrightarrow{(\widehat{f\circ g})_Y} Y$$
of the composite $\hat{f}_Y\circ \hat{g}_{f^{\ast}Y}$. The proof that $\iota_B$ and $\kappa_{f,g}$ indeed satisfy the associativity and unit axioms of pseudofunctors is given in  \cite[Lemma 10.4.7.]{NY}. In turn,  the coherence theorem for pseudofunctors (cf. \cite[Theorem 3.2.7 ]{ML}) guarantees that any two composites of canonical isomorphisms of the form \eqref{canonical_isos_pseudo} (and their inverses) with the same source and target are equal.

\smallskip

If a fibration $p:{\cal E}\rightarrow {\cal C}$ is also an opfibration, we can, by dual reasoning, fix a {\em cocleavage  over} ${\cal C}$ with respect to $p$  by choosing for each  morphism $f:B\rightarrow A$   in ${\cal C}$ and $Y\in {\it Ob}({\cal E}_B)$ a {\em cocartesian lift} $\tilde{f}_Y: Y\rightarrow f_{!}Y$ of $f$ at $Y$. This choice induces in the obvious way a direct image functor $f_!:{\cal E}_B\rightarrow {\cal E}_A$.

\smallskip
In what follows, we shall always assume that a bifibration is {\em cloven}, i.e., that it comes equipped with a chosen cleavage and cocleavage. 

\smallskip
 
Let  $p:{\cal E}\rightarrow {\cal C}$ be a bifibration. For an arrow $f:B\rightarrow A$ of ${\cal C}$ and $X\in {\it Ob}({\cal E}_B)$, the
cartesian lift $\hat{f}_{f_!X}:f^*f_!X\to f_!X$ determines a unique
factorization
\[
X \xrightarrow{\eta_{f,X}} f^*f_!X \xrightarrow{\hat{f}_{f_!X}} f_!X,
\]
of $\tilde{f}_X$ and the components $\eta_{f,X}$ assemble into a natural transformation
$\eta_f:{\text{Id}}_{{\cal E}_B}\Rightarrow f^{\ast}f_{!} $.
Dually,  for  $Y\in{\it Ob}({\cal E}_A)$,  the cocartesian lift
$\tilde{f}_{f^*Y}:f^{\ast}Y\to f_!f^*Y$ determines a unique factorization
\[
f^{\ast}Y\xrightarrow{\tilde{f}_{f^*Y}}f_!f^*Y \xrightarrow{\varepsilon_{f,Y}} Y,
\] of $\hat{f}_Y$,
yielding a natural transformation $\varepsilon_f:f_{!}f^{\ast}\Rightarrow {\text{Id}}_{{\cal E}_A}$.
The natural transformations $\eta_f$ and $\varepsilon_f$ satisfy the triangle identities by universal property of 
 cartesian and cocartesian arrows, so the pullback and the direct image functor form an adjunction $f^{\ast}\dashv f_{!}:{{\cal E}_B}\rightarrow {{\cal E}_A}$ with unit $\eta_f$ and counit $\varepsilon_f$.

\smallskip

A bifibration $p:{\cal E}\rightarrow{\cal C}$ satisfies the {\em Beck-Chevalley property} if, for any pullback square 
$$P=
\raisebox{-0.85cm}{\begin{tikzpicture}
\node(a) at (0,0) {$C$};\node(b) at (1,0) {$D$};\node(c) at (0,1) {$A$};\node(d) at (1,0.975) {$B$};
\draw[->](a)--(b);\draw[->] (c)--(0.8,1) ;\draw[->](1,0.8)--(b);\draw[->](c)--(a);
\node(a') at (0.5,1.2) {\small $a$};
\node(b') at (0.5,-0.2) {\small $b$};
\node(c') at (1.2,0.5) {\small $d$};
\node(d') at (-0.2,0.5) {\small $c$};
\end{tikzpicture}}$$
in ${\cal C}$, the  {\em Beck-Chevalley transformation} ${\mathsf{BC}}_P:c_!  a^{\ast}\rightarrow b^{\ast}  d_!$, defined by $${\mathsf{BC}}_P:=(\varepsilon_c \cdot b^{\ast}  d_!)\circ(c_!\cdot (\kappa^{-1}_{b,c}\circ\kappa_{d,a})\cdot d_!)\circ (c_!  a^{\ast}\cdot \eta_d),$$ is an isomorphism.

\begin{remark}[Fibration-first]
The condition "all pullback functors have left adjoints" is precisely what promotes a fibration to a bifibration (cf. \cite[Lemma 2.2.8.]{LR}). In Section \ref{syncat}, when encoding the pseudofunctorial structure of a  bifibration    by a string-diagrammatic 2-category,  we shall adopt this {\em fibration-first} perspective: the reindexing functors $f^{\ast}:{\cal E}_A\rightarrow {\cal E}_B$ will be taken as primitive. The direct image functors $f_!$ will exist only as left adjoints to $f^{\ast}$, and their pseudofunctorial structure, including Beck-Chevalley transformations, will be obtained via mate calculus from the pullback side.  This keeps the structure minimal  by avoiding redundant assumptions on $f_!$. \hfill$\square$
\end{remark}

\subsection{String diagrams in category theory}
\noindent
String diagrams provide a graphical representation of morphisms in various kinds of categories  -- and more generally, of 2-cells in 2-categories or bicategories. They were introduced independently by Hotz \cite{Ho} and Joyal and Street \cite{JS88}, \cite{JoyalStreet1991}, with the goal of replacing long algebraic expressions by intuitive geometric objects. In a string diagram, structural equalities of morphisms correspond to simple topological deformations of the diagram, making otherwise opaque algebraic identities visually clear.

\smallskip

As mathematical objects, string diagrams are finite topological graphs embedded into a square, with free
 endpoints  of edges (or strings) bound to the top or bottom side of the square, together with
a   labelling specified by an appropriate signature $\Sigma$, considered up to boundary-- and label--preserving planar isotopy. In addition, one requires that the projection of any edge on a vertical side of the square is injective. One can then interpret any string diagram as a morphism (or a 2-cell) in the free structure generated by $\Sigma$, and vice versa, and  the soundness and completeness theorem  \cite[Theorem 3]{JS88} allows us to completely  switch to string diagrams when working with those morphisms (or 2-cells). In other words, the appropriate free structure can be   characterized  in terms of string diagrams and their  vertical and horizontal pasting.

\smallskip 

In the main part of this article, we shall use string diagrams as a graphical representation of 2-cells in 2-categories (concretely, of natural transformations in certain sub-2-categories of ${\mathsf{Cat}}$). In this setting, string diagrams are Poincar\' e dual to the standard formalism of pasting diagrams.  This duality is illustrated in Figure \ref{dualnost}.

\begin{center}
\begin{figure}[H]
\begin{tikzpicture}[scale=1.15]
\node(a) [circle, fill=black, draw=black, inner sep=0.3mm] at (-1.2,0) {};
\node(b) [circle, fill=black, draw=black,inner sep=0.3mm] at (1.2,0) {};
\node(A)  at (-1.4,0) {\small $A$};
\node(B)  at (1.4,0) {\small $B$};
\node(C)  at (-0.62,0.8) {\small $C$};
\node(D)  at (0.62,0.8) {\small $D$};
\node(E)  at (0,-0.9) {\small $E$};
\node(a1) [circle, fill=black, draw=black,inner sep=0.3mm] at (-0.6,0.6) {};
\node(a2) [circle, fill=black, draw=black,inner sep=0.3mm] at (0.6,0.6) {};
\node(b1) [circle, fill=black, draw=black,inner sep=0.3mm] at (0,-0.675) {};
\node(f1) [draw=none] at (-1.2,0.5) {\small $f_1$};\node(f2) [draw=none] at (0,0.9) {\small $f_2$};\node(fm) [draw=none] at (1.25,0.5) {\small $f_3$};
\node(g1) [draw=none] at (-0.9,-0.7) {\small $g_1$};\node(gn) [draw=none] at (0.9,-0.7) {\small $g_2$};
\node(al) [draw=none] at (0.2,0) {$\alpha$};
\draw[thick,->] (-1.2,0.1) to[out=70,in=200] (-0.7,0.6); 
\draw[thick,->] (-1.2,-0.1) to[out=-70,in=180] (-0.1,-0.675); 
\draw[thick,->] (0.7,0.6) to[out=-20,in=120] (1.2,0.1); 
\draw[thick,->] (0.1,-0.675) to[out=0,in=-120] (1.2,-0.1); 
\draw[thick,->] (-0.5,0.65) to[out=10,in=170] (0.5,0.65); 
\draw[-Implies,line width=0.7pt,double distance=1.5pt](0,0.4)--(0,-0.4);
\end{tikzpicture} \quad \raisebox{1cm}{$\leadsto$} \quad \raisebox{-0.2cm}{\begin{tikzpicture}
\draw[draw=none,fill=EggShell!60](-1,2)--(1,2)--(1,0)--(-1,0)--cycle;
\node(A)  at (-0.65,1) {\small $A$};
\node(B)  at (0.65,1) {\small $B$};
\node(C)  at (-0.3,1.65) {\small $C$};
\node(D)  at (0.3,1.65) {\small $D$};
\node(E)  at (0,0.35) {\small $E$};
\node(al) [circle,fill=white,draw=black,inner sep=0.4mm] at (0,1) {\small $\alpha$};
\draw[thick] (-0.8,2) to[out=-90,in=135] (al);\draw[thick] (0.8,2) to[out=-90,in=45] (al);\draw[thick] (0,2) to[out=-90,in=90] (al);
\draw[thick] (-0.5,0) to[out=90,in=-120] (al);\draw[thick] (0.5,0) to[out=90,in=-60] (al);
\node(f1) [draw=none] at (-0.8,2.2) {\small $f_1$};\node(f2) [draw=none] at (0,2.2) {\small $f_2$};\node(fm) [draw=none] at (0.8,2.2) {\small $f_3$};
\node(g1) [draw=none] at (-0.5,-0.2) {\small $g_1$};\node(gn) [draw=none] at (0.5,-0.2) {\small $g_2$};
\end{tikzpicture}}
\caption{An example of Poincar\' e duality between pasting diagrams   and string diagrams: 0-cells $A$, $B$, $C$, $D$ and $E$, and the 2-cell $\alpha:f_3 f_2  f_1\Rightarrow g_2  g_1$, swap their dimension. In the string-diagram representation, $\alpha$  is given by a vertex in the plane whose input strings
 correspond  to factors of its domain, and whose output strings 
correspond to the factors of its codomain; in particular, a diagram is  read   top-down. The 0-cells are identified with the 2-dimensional regions delimited by the strings.}
\label{dualnost}
\end{figure}
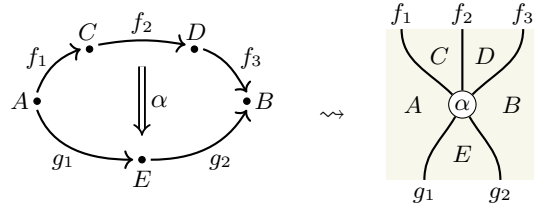
\end{center}
\begin{convention}
In what follows, instead of labeling explicitly the 2-dimensional regions of a string diagram by categories,   we shall use different colours  to denote different categories. The 
identification of a colour with a particular category will be given  only in Section 4, for the proof of the B\' enabou-Roubaud theorem.
\end{convention}
Vertical  composition of 2-cells is given by vertical {\em top-to-bottom} pasting of string diagrams, and horizontal composition of 2-cells is given by horizontal {\em left-to-right} pasting of string diagrams. The interchange law corresponds  to the fact that diagrams of the form
\begin{equation}
\raisebox{-1.2cm}{\begin{tikzpicture}
\draw[draw=none,fill=EggShell!60](-1,2)--(1,2)--(1,0)--(-1,0)--cycle;
\draw[draw=none,fill=DustyRed!40](-0.5,0)--(0.5,0)--(0.5,2)--(-0.5,2)--cycle;
\draw[draw=none,fill=SlateBlue!40] (0.5,0)--(0.5,2)--(1,2)--(1,0)--cycle;
\node(a) [circle,fill=white,draw=black,inner sep=0.45mm] at (-0.5,1.4) {\small $\alpha$};
\node(b) [circle,fill=white,draw=black,inner sep=0.2mm] at (-0.5,0.6) {\footnotesize $\beta$};
\node(c) [circle,fill=white,draw=black,inner sep=0.35mm] at (0.5,1.4) {\small $\gamma$};
\node(d) [circle,fill=white,draw=black,inner sep=0.4mm] at (0.5,0.6) {\small $\delta$};
\draw[thick] (-0.5,2)--(a)--(b)--(-0.5,0);\draw[thick] (0.5,2)--(c)--(d)--(0.5,0);
\node(f1) [draw=none] at (-0.5,2.2) {\small $f_1$};\node(h1) [draw=none] at (-0.5,-0.2) {\small $h_1$};
\node(f2) [draw=none] at (0.5,2.2) {\small $f_2$};\node(h1) [draw=none] at (0.5,-0.2) {\small $h_2$};
\node(g1) [draw=none] at (-0.7,1) {\small $g_1$};\node(g2) [draw=none] at (0.75,1) {\small $g_2$};
\end{tikzpicture}}\label{interchange}
\end{equation}
are well-defined. 

\smallskip

 One of the strengths of the string diagrammatic calculus is that it allows us to deal with the identities implicitly: identity 1-cells appear as ``invisible'' strings (i.e., as ``unbroken" regions), and identity 2-cells appear as ``invisible'' nodes (i.e., as bare strings). For example, the identity 1-cell ${\text{id}}_A$ and the identity 2-cell ${\text{Id}}_f$, for $f:A\rightarrow B$, are presented as
\begin{center}
\raisebox{0.5cm}{\begin{tikzpicture}
\draw[draw=none,fill=EggShell!60](0,0)--(1,0)--(1,1)--(0,1)--cycle;
\end{tikzpicture}} \quad\quad\raisebox{0.85cm}{\mbox{and}}\quad\quad \begin{tikzpicture}
\draw[draw=none,fill=EggShell!60](0,0)--(1,0)--(1,1)--(0,1)--cycle;
\draw[draw=none,fill=DustyRed!40](0.5,0)--(1,0)--(1,1)--(0.5,1)--cycle;
\draw[thick] (0.5,1)--(0.5,0);
\node at (0.5,-0.2) {\small $f$};\node at (0.5,1.2) {\small $f$};
\end{tikzpicture}
\end{center}
respectively. By substituting in \eqref{interchange} $\beta$ and $\gamma$ by the appropriate identity 2-cells,  we get the {\em sliding equalities}
\begin{equation}
\raisebox{-1.2cm}{\begin{tikzpicture}
\draw[draw=none,fill=EggShell!60](-1,2)--(1,2)--(1,0)--(-1,0)--cycle;
\draw[draw=none,fill=DustyRed!40](-0.5,0)--(0.5,0)--(0.5,2)--(-0.5,2)--cycle;
\draw[draw=none,fill=SlateBlue!40] (0.5,0)--(0.5,2)--(1,2)--(1,0)--cycle;
\node(a) [circle,fill=white,draw=black,inner sep=0.45mm] at (-0.5,1.4) {\small $\alpha$};
\node(d) [circle,fill=white,draw=black,inner sep=0.4mm] at (0.5,0.6) {\small $\delta$};
\draw[thick] (-0.5,2)--(a)--(-0.5,0);\draw[thick] (0.5,2)--(d)--(0.5,0);
\node(f1) [draw=none] at (-0.5,2.2) {\small $f_1$};\node(h1) [draw=none] at (-0.5,-0.2) {\small $g_1$};
\node(f2) [draw=none] at (0.5,2.2) {\small $g_2$};\node(h1) [draw=none] at (0.5,-0.2) {\small $h_2$};
\end{tikzpicture} \quad\raisebox{1.3cm}{$=$}\quad \begin{tikzpicture}
\draw[draw=none,fill=EggShell!60](-1,2)--(1,2)--(1,0)--(-1,0)--cycle;
\draw[draw=none,fill=DustyRed!40](-0.5,0)--(0.5,0)--(0.5,2)--(-0.5,2)--cycle;
\draw[draw=none,fill=SlateBlue!40] (0.5,0)--(0.5,2)--(1,2)--(1,0)--cycle;
\node(a) [circle,fill=white,draw=black,inner sep=0.45mm] at (-0.5,1) {\small $\alpha$};
\node(d) [circle,fill=white,draw=black,inner sep=0.4mm] at (0.5,1) {\small $\delta$};
\draw[thick] (-0.5,2)--(a)--(-0.5,0);\draw[thick] (0.5,2)--(d)--(0.5,0);
\node(f1) [draw=none] at (-0.5,2.2) {\small $f_1$};\node(h1) [draw=none] at (-0.5,-0.2) {\small $g_1$};
\node(f2) [draw=none] at (0.5,2.2) {\small $g_2$};\node(h1) [draw=none] at (0.5,-0.2) {\small $h_2$};
\end{tikzpicture}\quad\raisebox{1.3cm}{$=$}\quad \begin{tikzpicture}
\draw[draw=none,fill=EggShell!60](-1,2)--(1,2)--(1,0)--(-1,0)--cycle;
\draw[draw=none,fill=DustyRed!40](-0.5,0)--(0.5,0)--(0.5,2)--(-0.5,2)--cycle;
\draw[draw=none,fill=SlateBlue!40] (0.5,0)--(0.5,2)--(1,2)--(1,0)--cycle;
\node(a) [circle,fill=white,draw=black,inner sep=0.45mm] at (-0.5,0.6) {\small $\alpha$};
\node(d) [circle,fill=white,draw=black,inner sep=0.4mm] at (0.5,1.4) {\small $\delta$};
\draw[thick] (-0.5,2)--(a)--(-0.5,0);\draw[thick] (0.5,2)--(d)--(0.5,0);
\node(f1) [draw=none] at (-0.5,2.2) {\small $f_1$};\node(h1) [draw=none] at (-0.5,-0.2) {\small $g_1$};
\node(f2) [draw=none] at (0.5,2.2) {\small $g_2$};\node(h1) [draw=none] at (0.5,-0.2) {\small $h_2$};
\end{tikzpicture}}\label{sliding}
\end{equation}
which show how horizontal composition can be treated indifferently as the vertical composition  $(g_1\cdot \delta)\circ(\alpha\cdot g_2)$ or the vertical composition $(\alpha\cdot h_2)\circ (f_1\cdot \delta)$.

\smallskip

When proving B\' enabou-Roubaud theorem,  we will need to reason not only about categories, functors and natural transformations, but, at the same time,  about 
``ordinary'' objects and morphisms in those categories. This is accomplish by the standard technique of  identifying the objects of a category ${\cal C}$ with functors   $1\rightarrow {\cal C}$, where $1$ is the terminal category.     A morphism $f:X\rightarrow Y$ is then a natural transformation between
 two such  functors, and hence it is represented by a vertex in the string-diagrammatic calculus. With this convention, by replacing $\alpha$ by $f$ in the sliding equalities \eqref{sliding}, we see how naturality equations are ``handled by the syntax'' in the formalism of string diagrams.

\section{The string-diagrammatic 2-category $\mathsf{BiFib}_{\cal C}$}\label{syncat}
\noindent In this section, for a given category ${\cal C}$, we introduce a string-diagrammatic 2-category $\mathsf{BiFib}_{\cal C}$ that syntactically encodes the data of pseudofunctoriality of $f\mapsto f^{\ast}$ and the data of the adjunction $f^{\ast}\dashv f_!:{\cal E}_B\rightarrow {\cal E}_A$, for all morphisms $f:B\rightarrow A$ in ${\cal C}$, together with the (invertible) Beck-Chevalley transformations defined for pullbacks in ${\cal C}$. We proceed in two stages. First we define the 2-category  $\mathsf{Fib}_{\cal C}$ that captures only  the pseudofunctoriality of $(-)^*:  \mathcal C^{\mathrm{op}}\rightarrow \mathsf{Cat}
$. We then extend  $\mathsf{Fib}_{\cal C}$ to the full bifibrational 2-category $\mathsf{BiFib}_{\cal C}$.

\subsection{The  2-category $\mathsf{Fib}_{\cal C}$ encoding pseudofunctoriality}
 
\noindent   A {\em path} in a category ${\cal C}$ is a sequence 
$${\bf f} := B_n \xrightarrow{f_n} B_{n-1} \xrightarrow{f_{n\!-\!1}} \cdots \xrightarrow{f_1} B_0, \quad n\geq 1,$$
of composable arrows in ${\cal C}$. The object $B_n$ (resp. $B_0$) is called the source (resp. the target) of ${\bf f}$.  If two paths ${\bf f}$ and ${\bf g}$ have the same source and target, we say that they are parallel. The {\em value of a path} ${\bf f}$, denoted by $[f]$, is the morphism in ${\cal C}$ obtained by composing the arrows of ${\bf f}$.
We extend the notion of  a path by allowing, for each object $B\in {\it Ob}({\cal C})$,   the {\em empty path on} $B$, which we denote by $\emptyset_B$; the source and the target of $\emptyset_B$ is $B$. When a path ${\bf f}$ in ${\cal C}$ consists of a single arrow   $f$,  we shall denote it simply by $f$. For  paths ${\bf f}:A\rightarrow B$ and ${\bf g}:B\rightarrow C$, we shall write ${\bf f}\cdot {\bf g}:A\rightarrow C$ for the path obtained by concatenating $f$ and $g$ in the obvious way; concatenation of paths is associative on the nose and has empty paths as neutral elements.

\smallskip

 A {\em (proper) polygon type} in ${\cal C}$ is a pair $({\bf f},{\bf g})$ of   parallel  paths 
$$
{\bf f} := B_n \xrightarrow{f_n} B_{n-1} \xrightarrow{f_{n\!-\!1}} \cdots \xrightarrow{f_1} B_0
\quad\mbox{ and }\quad
{\bf g}: = A_m \xrightarrow{g_m} A_{m-1} \xrightarrow{g_{m\!-\!1}} \cdots \xrightarrow{g_1} A_0, 
$$
  $n,m\ge 1$, in ${\cal C}$, such that $[{\bf f}]=[{\bf g}]$.  The object $B_n$ (resp.  $B_0$) is called the source  (resp. the target) of  $({\bf f},{\bf g})$. Polygon types record the outer boundaries of commutative polygons in ${\cal C}$ (that is,
of commutative diagrams in ${\cal C}$ embedded in the plane for which the designation of “top”
and “bottom” is given to the two outer parallel paths).  In particular, while there can be many different commutative polygons witnessing the equality $[{\bf f}]=[{\bf g}]$, they all share the same polygon type $({\bf f},{\bf g})$. We extend the notion of   polygon type by allowing, for each object $B\in {\it Ob}({\cal C})$,   polygon types of the form $({\bf f},\emptyset_B)$ and, symmetrically, of the form $(\emptyset_B,{\bf f})$, where ${\bf f}:B\rightarrow B$ is any path in ${\cal C}$ such that $[{\bf f}]={\mathrm{id}}_B$; we refer to such polygon types as  {\em degenerate polygon types}. In addition, we also consider $(\emptyset_B,\emptyset_B)$ as a degenerate polygon type. 
The {\em opposite} polygon type of a (possibly degenerate) polygon type $({\bf f}, {\bf g})$ is the polygon type
$({\bf f}, {\bf g})^{\it op} := ({\bf g},{\bf f})$.  Polygon types in ${\cal C}$ naturally admit the operations of vertical   and horizontal pasting:
\begin{itemize}
\item  vertical pasting of polygon types $({\bf f},{\bf g})$ and $({\bf g},{\bf h})$ with matching intermediate boundary   ${\bf g}$ (which might be empty) is defined by $({\bf g},{\bf h})\circ ({\bf f},{\bf g}):=({\bf f},{\bf h})$;
\item   horizontal pasting of polygon types $({\bf f}, {\bf g})$ and $({\bf h}, {\bf k})$, such that the target object of $({\bf f}, {\bf g})$ matches the source object of $({\bf h}, {\bf k})$, is defined by $({\bf f}, {\bf g})\cdot ({\bf h}, {\bf k}):=({\bf f}\cdot {\bf h},{\bf g}\cdot {\bf k})$.
\end{itemize}
Both of these operations are clearly strictly associative and unital.

\begin{definition}[The pseudofunctorial  signature $\Sigma^{\ast}_{{\cal C}}$]\label{def:chi-signature}
Let ${\cal C}$ be a category. The {\em  pseudofunctorial  signature} $\Sigma^{\ast}_{{\cal C}}$  is the  2-polygraph defined as follows:

\begin{itemize}

\item  {0-cells: $(\Sigma^{\ast}_{{\cal C}})_0:=\{{\cal E}_A\,|\, A\in{\it Ob}({\cal C})\}$}; 
 \item  {1-cells:   $(\Sigma^{\ast}_{{\cal C}})_1:=\{f^{\ast}:{\cal E}_B\rightarrow {\cal E}_A\,|\, f:A\rightarrow B \in {\it Arr}({\cal C})\}$}; 
\item 2-cells: For each  proper polygon type $({\bf f}, {\bf g})$ in ${\cal C}$,
where
${\bf f} = X_n \xrightarrow{f_n} X_{n-1} \xrightarrow{f_{n\!-\!1}} \cdots \xrightarrow{f_1} X_0$, and
$
{\bf g} = Y_m \xrightarrow{g_m} Y_{m-1} \xrightarrow{g_{m\!-\!1}} \cdots \xrightarrow{g_1} Y_0
$,  we
adjoin a 2-cell
\[
\chi_{({\bf f}, {\bf g})}:
\ f_n^{*} \cdots   f_1^{*}\;\Rightarrow\; g_m^{*}  \cdots   g_1^{*};
\]
for degenerate polygon types $({\bf f},\emptyset_B)$, $(\emptyset_B,{\bf f})$ and $(\emptyset_B,\emptyset_B)$, where $B\in {\cal C}$ and ${\bf f} = B \xrightarrow{f_n} X_{n-1} \xrightarrow{f_{n\!-\!1}} \cdots \xrightarrow{f_1} B$, we adjoin
$$
\chi_{({\bf f},\emptyset_B)}:  f_n^{*} \cdots   f_1^{*} \Rightarrow \epsilon_{{\cal E}_B},\quad
\chi_{(\emptyset_B,{\bf f})}:\epsilon_{{\cal E}_B}\Rightarrow   f_n^{*} \cdots   f_1^{*} \quad \mbox{ and } \quad \chi_{(\emptyset_B,\emptyset_B)}:\epsilon_{{\cal E}_B}\Rightarrow \epsilon_{{\cal E}_B},
$$ 
respectively, where $\epsilon_{{\cal E}_B}$ denotes the empty path of 1-cells at ${{\cal E}_B}$. By writing ${\bf f}^{\ast}$ instead of  $f_n^{*} \cdots   f_1^{*}$ (resp. instead of $\epsilon_{{\cal E}_B}$) if  ${\bf f} =X_n \xrightarrow{f_n} X_{n-1} \xrightarrow{f_{n\!-\!1}} \cdots \xrightarrow{f_1} X_0$ (resp.  if ${\bf f}=\emptyset_B$), we therefore have:  
$$ 
(\Sigma^{\ast}_{{\cal C}})_2 := \{\chi_{({\bf f},{\bf g})}:{\bf f}^{\ast}\Rightarrow {\bf g}^{\ast} \,|\, ({\bf f}, {\bf g}) \mbox{ is a  polygon type in } {\cal C}\}.
 $$
\end{itemize}\hfill $\square$
 \end{definition}
 \begin{definition}[The string-diagrammatic 2-category $\mathsf{Fib}_{\cal C}$]\label{fff}
We define the string-diagrammatic 2-category $\mathsf{Fib}_{\cal C}$ as the quotient
\[
 \mathsf{Fib}_{\cal C}
\ :=\
\mathcal{F}(\Sigma^{\ast}_{{\cal C}}) \big/ {\mathsf{Rel}^{\ast}}
\]
of the free strict 2-category generated by the pseudofunctorial signature $\Sigma^{\ast}_{{\cal C}}$ from Definition \ref{def:chi-signature}  under the congruence of $2$-cells ${\mathsf{Rel}^{\ast}}$ generated by the following relations:\\

\begin{center}
%

 \begin{tikzpicture}
\draw[draw=none,fill=EggShell!60](-1,2)--(1,2)-- (1,0)-- (-1,0)--(-1,2);
\node(a)[rectangle,draw=black,inner sep=0.5mm,fill=white,rounded corners=0.5mm]  at (0,1) {\footnotesize ${\chi}_{({\bf f},{\bf f})}$};
\draw[thick] (-0.5,2)to[out=-90,in=120](a)to[out=60,in=-90](0.5,2); 
\draw[thick] (-0.5,0)to[out=90,in=-120](a)to[out=-60,in=90](0.5,0); 
\node at (-0.5,2.2) {\scriptsize  $f_1^{\ast}$}; 
\node at (-0.5,-0.2) {\scriptsize  $f_1^{\ast}$}; \node at (0.5,2.2) {\scriptsize  $f_n^{\ast}$}; 
\node at (0.5,-0.2) {\scriptsize  $f_n^{\ast}$}; 
 \node at (0,-0.2) {\tiny $\cdot\!\!\cdot\!\!\cdot$}; \node at (0,2.2) {\tiny $\cdot\!\!\cdot\!\!\cdot$}; \node at (0,0.15) {\tiny $\cdot\!\!\cdot\!\!\cdot$}; \node at (0,1.85) {\tiny $\cdot\!\!\cdot\!\!\cdot$};
\end{tikzpicture}\enspace \raisebox{1.3cm}{$\overset{(1)}{=}$}\enspace \begin{tikzpicture}
\draw[draw=none,fill=EggShell!60](-1,2)--(1,2)-- (1,0)-- (-1,0)--(-1,2);
 \draw[thick] (-0.5,2)to[out=-90,in=90](-0.5,0); \draw[thick] (0.5,2)to[out=-90,in=90](0.5,0); 
\node at (-0.5,2.2) {\scriptsize  $f_1^{\ast}$}; 
\node at (-0.5,-0.2) {\scriptsize  $f_1^{\ast}$}; \node at (0.5,2.2) {\scriptsize  $f_n^{\ast}$}; 
\node at (0.5,-0.2) {\scriptsize  $f_n^{\ast}$}; 
 \node at (0,-0.2) {\tiny $\cdot\!\!\cdot\!\!\cdot$}; \node at (0,2.2) {\tiny $\cdot\!\!\cdot\!\!\cdot$}; \node at (0,1) {\tiny $\cdot\!\!\cdot\!\!\cdot$}; 
\end{tikzpicture} \quad\quad\quad\quad  \begin{tikzpicture}
\draw[draw=none,fill=EggShell!60](-1,2)--(1,2)-- (1,0)-- (-1,0)--(-1,2);
\node(a)[rectangle,draw=black,inner sep=0.5mm,fill=white,rounded corners=0.5mm]  at (0,1.35) {\footnotesize ${\chi}_{({\bf f},{\bf g})}$};
\node(b)[rectangle,draw=black,inner sep=0.5mm,fill=white,rounded corners=0.5mm]  at (0,0.65) {\footnotesize ${\chi}_{({\bf g},{\bf h})}$};
\draw[thick] (-0.5,2)to[out=-90,in=120](a)to[out=60,in=-90](0.5,2); 
\draw[thick] (-0.5,0)to[out=90,in=-120](b)to[out=-60,in=90](0.5,0);  
\draw[thick] (a) to[out=-120,in=120] (b) to[out=60,in=-60] (a);
\node at (-0.5,2.2) {\scriptsize  $f_1^{\ast}$}; 
\node at (-0.5,-0.2) {\scriptsize  $h_1^{\ast}$}; \node at (0.5,2.2) {\scriptsize  $f_n^{\ast}$}; 
\node at (0.5,-0.2) {\scriptsize  $h_m^{\ast}$}; 
 \node at (0,-0.2) {\tiny $\cdot\!\!\cdot\!\!\cdot$}; \node at (0,2.2) {\tiny $\cdot\!\!\cdot\!\!\cdot$}; \node at (0,0.15) {\tiny $\cdot\!\!\cdot\!\!\cdot$}; \node at (0,1.85) {\tiny $\cdot\!\!\cdot\!\!\cdot$}; \node at (0,1) {\tiny $\cdot\!\!\cdot\!\!\cdot$}; 
\end{tikzpicture}\enspace \raisebox{1.3cm}{$\overset{(2)}{=}$}\enspace  \begin{tikzpicture}
\draw[draw=none,fill=EggShell!60](-1,2)--(1,2)-- (1,0)-- (-1,0)--(-1,2);
\node(a)[rectangle,draw=black,inner sep=0.5mm,fill=white,rounded corners=0.5mm]  at (0,1) {\footnotesize ${\chi}_{({\bf f},{\bf h})}$};
\draw[thick] (-0.5,2)to[out=-90,in=120](a)to[out=60,in=-90](0.5,2); 
\draw[thick] (-0.5,0)to[out=90,in=-120](a)to[out=-60,in=90](0.5,0);  
\node at (-0.5,2.2) {\scriptsize  $f_1^{\ast}$}; 
\node at (-0.5,-0.2) {\scriptsize  $h_1^{\ast}$}; \node at (0.5,2.2) {\scriptsize  $f_n^{\ast}$}; 
\node at (0.5,-0.2) {\scriptsize  $h_m^{\ast}$}; 
 \node at (0,-0.2) {\tiny $\cdot\!\!\cdot\!\!\cdot$}; \node at (0,2.2) {\tiny $\cdot\!\!\cdot\!\!\cdot$}; \node at (0,0.15) {\tiny $\cdot\!\!\cdot\!\!\cdot$}; \node at (0,1.85) {\tiny $\cdot\!\!\cdot\!\!\cdot$}; 
\end{tikzpicture} 
\vspace{0.3cm}

\begin{tikzpicture}
\draw[draw=none,fill=EggShell!60](-1,2)--(1,2)-- (1,0)-- (-1,0)--(-1,2);
\node(a)[rectangle,draw=black,inner sep=0.5mm,fill=white,rounded corners=0.5mm]  at (0.1,1) {\footnotesize ${\chi}_{({\bf f},{\bf g})}$};
\draw[thick] (-0.75,0) to[out=90,in=-90] (-0.75,2);
\draw[thick] (-0.4,2) to[out=-90,in=120] (a) to[out=60,in=-90] (0.6,2);
\draw[thick] (-0.4,0) to[out=90,in=-120] (a) to[out=-60,in=90] (0.6,0);
\node at (-0.75,2.2) {\scriptsize $a^{\ast}$};\node at (-0.4,2.2) {\scriptsize  $f_1^{\ast}$};\node at (0.1,2.2)  {\tiny $\cdot\!\!\cdot\!\!\cdot$}; \node at (0.1,-0.2)  {\tiny $\cdot\!\!\cdot\!\!\cdot$};\node at (0.6,2.2) {\scriptsize $f_n^{\ast}$};
\node at (-0.75,-0.2) {\scriptsize  $a^{\ast}$};\node at (-0.4,-0.23) {\scriptsize $g_1^{\ast}$};\node at (0.6,-0.23) {\scriptsize  $g_m^{\ast}$};
\node at (0.1,1.85)  {\tiny $\cdot\!\!\cdot\!\!\cdot$}; \node at (0.1,0.15)  {\tiny $\cdot\!\!\cdot\!\!\cdot$}; 
\end{tikzpicture} 
\enspace \raisebox{1.3cm}{$\overset{(3.1)}{=}$}\enspace  \begin{tikzpicture}
\draw[draw=none,fill=EggShell!60](-1,2)--(1,2)-- (1,0)-- (-1,0)--(-1,2);
\node(a)[rectangle,draw=black,inner sep=0.5mm,fill=white,rounded corners=0.5mm]  at (0,1) {\footnotesize ${\chi}_{({\bf f}\cdot a,{\bf g}\cdot a)}$};
\draw[thick] (-0.6,2)to[out=-90,in=120](a)to[out=100,in=-90](-0.3,2); 
\draw[thick] (0.6,2) to[out=-90,in=60] (a);
\draw[thick] (-0.6,0)to[out=90,in=-120](a)to[out=-100,in=90](-0.3,0); 
\draw[thick] (0.6,0) to[out=90,in=-60] (a);
\node at (-0.6,2.2) {\scriptsize  $a^{\ast}$};\node at (0.6,2.2) {\scriptsize   $f_n^{\ast}$};\node at (-0.3,2.2) {\scriptsize   $f_1^{\ast}$};
\node at (-0.6,-0.2) {\scriptsize  $a^{\ast}$};\node at (0.6,-0.2) {\scriptsize   $g_m^{\ast}$};\node at (-0.3,-0.2) {\scriptsize   $g_1^{\ast}$};\node at (0.1,2.2) {\tiny $\cdot\!\!\cdot\!\!\cdot$};\node at (0.1,-0.2) {\tiny $\cdot\!\!\cdot\!\!\cdot$};\node at (0.1,1.85) {\tiny $\cdot\!\!\cdot\!\!\cdot$};\node at (0.1,0.15) {\tiny $\cdot\!\!\cdot\!\!\cdot$};
\end{tikzpicture}
\quad\quad\quad\quad
\begin{tikzpicture}
\draw[draw=none,fill=EggShell!60](-1,2)--(1,2)-- (1,0)-- (-1,0)--(-1,2);
\node(a)[rectangle,draw=black,inner sep=0.5mm,fill=white,rounded corners=0.5mm]  at (-0.1,1) {\footnotesize ${\chi}_{({\bf f},{\bf g})}$};
\draw[thick] (0.75,0) to[out=90,in=-90] (0.75,2);
\draw[thick] (0.3,2) to[out=-90,in=60] (a) to[out=120,in=-90] (-0.5,2);\draw[thick] (-0.5,0) to[out=90,in=-120] (a) to[out=-60,in=90] (0.3,0);
\node at (0.75,2.2) {\scriptsize  $b^{\ast}$};\node at (0.3,2.2) {\scriptsize $f_n^{\ast}$};\node at (-0.5,2.2) {\scriptsize $f_1^{\ast}$}; \node at (-0.1,2.2) {\tiny $\cdot\!\!\cdot\!\!\cdot$};\node at (-0.1,1.85) {\tiny $\cdot\!\!\cdot\!\!\cdot$};
\node at (0.75,-0.2) {\scriptsize $b^{\ast}$};\node at (0.3,-0.2) {\scriptsize $g_m^{\ast}$};\node at (-0.5,-0.2) {\scriptsize  $g_1^{\ast}$}; \node at (-0.1,-0.2) {\tiny $\cdot\!\!\cdot\!\!\cdot$};  \node at (-0.1,0.15) {\tiny $\cdot\!\!\cdot\!\!\cdot$}; 
\end{tikzpicture} \enspace \raisebox{1.3cm}{$\overset{(3.2)}{=}$}\enspace \begin{tikzpicture}
\draw[draw=none,fill=EggShell!60](-1,2)--(1,2)-- (1,0)-- (-1,0)--(-1,2);
\node(a)[rectangle,draw=black,inner sep=0.5mm,fill=white,rounded corners=0.5mm]  at (0,1) {\footnotesize ${\chi}_{(b\cdot{\bf f},b\cdot {\bf g})}$};
\draw[thick] (0.6,2)to[out=-90,in=60](a)to[out=80,in=-90](0.3,2); 
\draw[thick] (-0.6,2) to[out=-90,in=120] (a);
\draw[thick] (-0.6,0)to[out=90,in=-120](a)to[out=-80,in=90](0.3,0); 
\draw[thick] (0.6,0) to[out=90,in=-60] (a);
\node at (-0.6,2.2) {\scriptsize $f_1^{\ast}$};\node at (0.3,2.2) {\scriptsize $f_n^{\ast}$};
\node at (-0.6,-0.2) {\scriptsize $g_1^{\ast}$};\node at (0.3,-0.2) {\scriptsize $g_m^{\ast}$};
 \node at (-0.2,2.2) {\tiny $\cdot\!\!\cdot\!\!\cdot$}; \node at (-0.2,-0.2) {\tiny $\cdot\!\!\cdot\!\!\cdot$};
 \node at (-0.2,1.85) {\tiny $\cdot\!\!\cdot\!\!\cdot$}; \node at (-0.2,0.15) {\tiny $\cdot\!\!\cdot\!\!\cdot$};
\node at (0.66,2.2) {\scriptsize $b^{\ast}$};\node at (0.66,-0.2) {\scriptsize $b^{\ast}$};
\end{tikzpicture}
\end{center}

\noindent which apply to all polygon types,  including degenerate ones. In particular, in (1),  $({\bf f},{\bf f})$ is allowed to be $(\emptyset_B, \emptyset_B)$ for some $B\in {\it Ob}({\cal C})$ (in which case the rule, whose display no longer involves explicit strings, declares $\chi_{(\emptyset_B,\emptyset_B)}$ to be the identity 2-cell  ${\mathrm{Id}}_{{\mathrm{Id}}_{{\cal E}_B}}$), and, in (2), the common path ${\bf g}$ is  allowed to be empty as well (in which case there are no strings connecting the two vertices on the left hand side of the rule). We refer to  (1), (2), (3.1) and (3.2)   as  {\em identity}, {\em vertical pasting}, {\em left whiskering}  and {\em right whiskering}, respectively. \hfill$\square$
\end{definition}
%

The following lemma exhibits certain equalities of string diagrams which are easily derived from the generating relations.
\begin{lemma}\label{easy}
The following equalities of string diagrams hold in  {${\mathsf{Fib}}_{\cal C}$}:
\begin{itemize}
\item[{\it{a)}}] {$\chi$-invertibility}:
\begin{center} 
\raisebox{-1.4cm}{\begin{tikzpicture}
\draw[draw=none,fill=EggShell!60](-1,2)--(1,2)-- (1,0)-- (-1,0)--(-1,2);
\node(a)[rectangle,draw=black,inner sep=0.5mm,fill=white,rounded corners=0.5mm]  at (0,1.35) {\footnotesize ${\chi}_{({\bf f},{\bf g})}$};
\node(b)[rectangle,draw=black,inner sep=0.5mm,fill=white,rounded corners=0.5mm]  at (0,0.65) {\footnotesize ${\chi}_{({\bf f},{\bf g})^{\it op}}$};
\draw[thick] (-0.5,2)to[out=-90,in=120](a)to[out=60,in=-90](0.5,2);
 \draw[thick] (-0.5,0)to[out=90,in=-120](b)to[out=-60,in=90](0.5,0); 
\draw[thick] (a) to[out=-120,in=120] (b) to[out=60,in=-60] (a);
\node at (-0.5,2.2) {\scriptsize $f_1^{\ast}$};  \node at (0.5,2.2) {\scriptsize $f_n^{\ast}$};    \node at (0,-0.2) {\tiny $\cdot\!\!\cdot\!\!\cdot$};
\node at (-0.5,-0.2) {\scriptsize $f_1^{\ast}$};  \node at (0.5,-0.2) {\scriptsize$f_n^{\ast}$}; \node at (0,2.2) {\tiny $\cdot\!\!\cdot\!\!\cdot$};
\node at (0,0.15) {\tiny $\cdot\!\!\cdot\!\!\cdot$};\node at (0,1) {\tiny $\cdot\!\!\cdot\!\!\cdot$};\node at (0,1.85) {\tiny $\cdot\!\!\cdot\!\!\cdot$};
\end{tikzpicture}
\enspace \raisebox{1.3cm}{$=$}\enspace   \begin{tikzpicture}
\draw[draw=none,fill=EggShell!60](-1,2)--(1,2)-- (1,0)-- (-1,0)--(-1,2);
\draw[thick] (-0.5,2)to[out=-90,in=90](-0.5,0);  \draw[thick] (0.5,2)to[out=-90,in=90](0.5,0); 
\node at (-0.5,2.2) {\scriptsize $f_1^{\ast}$};  \node at (0.5,2.2) {\scriptsize $f_n^{\ast}$}; \node at (0,-0.2) {\tiny $\cdot\!\!\cdot\!\!\cdot$};
\node at (-0.5,-0.2) {\scriptsize $f_1^{\ast}$}; \node at (0.5,-0.2) {\scriptsize $f_n^{\ast}$}; \node at (0,2.2) {\tiny $\cdot\!\!\cdot\!\!\cdot$};
\node at (0,1) {\tiny $\cdot\!\!\cdot\!\!\cdot$}; 
\end{tikzpicture} \enspace \raisebox{1.3cm}{$=$}\enspace  
\begin{tikzpicture}
\draw[draw=none,fill=EggShell!60](-1,2)--(1,2)-- (1,0)-- (-1,0)--(-1,2);
\node(a)[rectangle,draw=black,inner sep=0.5mm,fill=white,rounded corners=0.5mm]  at (0,1.35) {\footnotesize ${\chi}_{({\bf g},{\bf f})^{\it op}}$};
\node(b)[rectangle,draw=black,inner sep=0.5mm,fill=white,rounded corners=0.5mm]  at (0,0.65) {\footnotesize ${\chi}_{({\bf g},{\bf f})}$};
\draw[thick] (-0.5,2)to[out=-90,in=120](a)to[out=60,in=-90](0.5,2);
 \draw[thick] (-0.5,0)to[out=90,in=-120](b)to[out=-60,in=90](0.5,0); 
\draw[thick] (a) to[out=-120,in=120] (b) to[out=60,in=-60] (a);
\node at (-0.5,2.2) {\scriptsize $f_1^{\ast}$};  \node at (0.5,2.2) {\scriptsize $f_n^{\ast}$};    \node at (0,-0.2) {\tiny $\cdot\!\!\cdot\!\!\cdot$};
\node at (-0.5,-0.2) {\scriptsize $f_1^{\ast}$};  \node at (0.5,-0.2) {\scriptsize$f_n^{\ast}$}; \node at (0,2.2) {\tiny $\cdot\!\!\cdot\!\!\cdot$};
\node at (0,0.15) {\tiny $\cdot\!\!\cdot\!\!\cdot$};\node at (0,1) {\tiny $\cdot\!\!\cdot\!\!\cdot$};\node at (0,1.85) {\tiny $\cdot\!\!\cdot\!\!\cdot$};
\end{tikzpicture}} \end{center}
\item[{\it{b)}}] {(generalized) whiskering}:
\begin{center}  \raisebox{-1.4cm}{\begin{tikzpicture}
\draw[draw=none,fill=EggShell!60](-1,2)--(1,2)-- (1,0)-- (-1,0)--(-1,2);
\node(a)[rectangle,draw=black,inner sep=0.4mm,fill=white,rounded corners=0.5mm]  at (0,1) {\footnotesize ${\chi}_{({\bf f},{\bf g})}$};
\draw[thick] (-0.2,2)to[out=-90,in=120](a)to[out=60,in=-90](0.2,2);
\draw[thick] (-0.2,0)to[out=90,in=-120](a)to[out=-60,in=90](0.2,0); 
\draw[thick](-0.9,2)--(-0.9,0);\draw[thick](-0.5,2)--(-0.5,0);\draw[thick](0.9,2)--(0.9,0);\draw[thick](0.5,2)--(0.5,0);
\node at (0,1.85)  {\tiny $\cdot\!\!\cdot\!\!\cdot$};\node at (-0.65,1.85)  {\tiny $\cdot\!\!\cdot\!\!\cdot$};\node at (0.65,1.85)  {\tiny $\cdot\!\!\cdot\!\!\cdot$};
 \node at (0,0.15)  {\tiny $\cdot\!\!\cdot\!\!\cdot$};\node at (-0.65,0.15)  {\tiny $\cdot\!\!\cdot\!\!\cdot$};\node at (0.65,0.15)  {\tiny $\cdot\!\!\cdot\!\!\cdot$};
\node at (-0.18,2.2) {\scriptsize $f_1^{\ast}$}; \node at (0,2.2)  {\tiny $\cdot\!\!\cdot\!\!\cdot$}; \node at (0.2,2.2) {\scriptsize  $f_n^{\ast}$};
\node at (-0.95,2.18) {\scriptsize  $h_1^{\ast}$}; \node at (-0.75,2.2)  {\tiny $\cdot\!\!\cdot\!\!\cdot$}; \node at (-0.5,2.18) {\scriptsize  $h_p^{\ast}$};
\node at (0.5,2.18) {\scriptsize  $k_1^{\ast}$}; \node at (0.75,2.2)  {\tiny $\cdot\!\!\cdot\!\!\cdot$}; \node at (0.965,2.18) {\scriptsize  $k_q^{\ast}$};
\node at (-0.95,-0.2) {\scriptsize  $h_1^{\ast}$}; \node at (-0.75,-0.2)  {\tiny $\cdot\!\!\cdot\!\!\cdot$}; \node at (-0.5,-0.2) {\scriptsize  $h_p^{\ast}$};
\node at (0.55,-0.2) {\scriptsize  $k_1^{\ast}$}; \node at (0.75,-0.2)  {\tiny $\cdot\!\!\cdot\!\!\cdot$}; \node at (0.965,-0.2) {\scriptsize  $k_q^{\ast}$};
\node at (-0.18,-0.2) {\scriptsize $g_1^{\ast}$}; \node at (0,-0.2)  {\tiny $\cdot\!\!\cdot\!\!\cdot$}; \node at (0.25,-0.2) {\scriptsize  $g_m^{\ast}$};
\end{tikzpicture}\enspace \raisebox{1.3cm}{$=$}\enspace  \begin{tikzpicture}
\draw[draw=none,fill=EggShell!60](-1,2)--(1,2)-- (1,0)-- (-1,0)--(-1,2);
\node(a)[rectangle,draw=black,inner sep=0.5mm,fill=white,rounded corners=0.5mm]  at (0,1) {\footnotesize ${\chi}_{({\bf k}\cdot{\bf f}\cdot{\bf h},{\bf k}\cdot{\bf g}\cdot{\bf h})}$};
\draw[thick] (-0.2,2)to[out=-90,in=110](a) to[out=70,in=-90](0.2,2); 
\draw[thick] (-0.2,0)to[out=90,in=-110](a)to[out=-70,in=90](0.2,0); 
 \draw[thick](-0.9,2)to[out=-90,in=140] (a) to[out=40,in=-90](0.9,2);
\draw[thick](-0.9,0)to[out=90,in=-140] (a) to[out=-40,in=90](0.9,0);
\draw[thick](-0.4,2)to[out=-90,in=120] (a) to[out=60,in=-90](0.4,2);
\draw[thick](-0.4,0)to[out=90,in=-120] (a) to[out=-60,in=90](0.4,0);
 \node at (0,1.85)  {\tiny $\cdot\!\!\cdot\!\!\cdot$};\node at (-0.65,1.85)  {\tiny $\cdot\!\!\cdot\!\!\cdot$};\node at (0.65,1.85)  {\tiny $\cdot\!\!\cdot\!\!\cdot$};
 \node at (0,0.15)  {\tiny $\cdot\!\!\cdot\!\!\cdot$};\node at (-0.65,0.15)  {\tiny $\cdot\!\!\cdot\!\!\cdot$};\node at (0.65,0.15)  {\tiny $\cdot\!\!\cdot\!\!\cdot$};
\node at (-0.18,2.2) {\scriptsize $f_1^{\ast}$}; \node at (0,2.2)  {\tiny $\cdot\!\!\cdot\!\!\cdot$}; \node at (0.2,2.2) {\scriptsize  $f_n^{\ast}$};
\node at (-0.95,2.18) {\scriptsize  $h_1^{\ast}$}; \node at (-0.75,2.2)  {\tiny $\cdot\!\!\cdot\!\!\cdot$}; \node at (-0.5,2.18) {\scriptsize  $h_p^{\ast}$};
\node at (0.5,2.18) {\scriptsize  $k_1^{\ast}$}; \node at (0.75,2.2)  {\tiny $\cdot\!\!\cdot\!\!\cdot$}; \node at (0.965,2.18) {\scriptsize  $k_q^{\ast}$};
\node at (-0.95,-0.2) {\scriptsize  $h_1^{\ast}$}; \node at (-0.75,-0.2)  {\tiny $\cdot\!\!\cdot\!\!\cdot$}; \node at (-0.5,-0.2) {\scriptsize  $h_p^{\ast}$};
\node at (0.55,-0.2) {\scriptsize  $k_1^{\ast}$}; \node at (0.75,-0.2)  {\tiny $\cdot\!\!\cdot\!\!\cdot$}; \node at (0.965,-0.2) {\scriptsize  $k_q^{\ast}$};
\node at (-0.18,-0.2) {\scriptsize $g_1^{\ast}$}; \node at (0,-0.2)  {\tiny $\cdot\!\!\cdot\!\!\cdot$}; \node at (0.25,-0.2) {\scriptsize  $g_m^{\ast}$};
\end{tikzpicture}}\end{center}
\item[{\it{c)}}] horizontal pasting:
\begin{center}
\raisebox{-1.4cm}{\begin{tikzpicture}
\draw[draw=none,fill=EggShell!60](-1,2)--(1,2)-- (1,0)-- (-1,0)--(-1,2);
\node(a)[rectangle,draw=black,inner sep=0.5mm,fill=white,rounded corners=0.5mm]  at (-0.2,1.35) {\footnotesize ${\chi}_{({\bf f},{\bf l}\cdot{\bf k})}$};
\node(b)[rectangle,draw=black,inner sep=0.5mm,fill=white,rounded corners=0.5mm]  at (0.2,0.65) {\footnotesize ${\chi}_{({\bf h}\cdot {\bf l},{\bf g})}$};
\draw[thick] (-0.5,2)to[out=-90,in=120](a)to[out=60,in=-90](0.1,2); 
 \draw[thick] (-0.1,0)to[out=90,in=-120](b)to[out=-60,in=90](0.5,0);  
\draw[thick] (a) to[out=-90,in=140] (b) to[out=90,in=-40] (a);
\draw[thick] (-0.75,0) to[out=90,in=-140] (a) to[out=-120,in=90](-0.4,0);  
\draw[thick] (0.75,2) to[out=-90,in=40] (b) to[out=60,in=-90](0.4,2);
\node at (-0.5,2.2) {\scriptsize  $f_1^{\ast}$}; \node at (0.05,2.2) {\scriptsize $f_n^{\ast}$};  \node at (-0.25,2.2)  {\tiny $\cdot\!\!\cdot\!\!\cdot$}; \node at (0.82,2.18) {\scriptsize   $h_p^{\ast}$}; \node at (0.35,2.2) {\scriptsize $h_1^{\ast}$};  \node at (0.55,2.2)  {\tiny $\cdot\!\!\cdot\!\!\cdot$}; 
\node at (-0.05,-0.2){\scriptsize   $g_1^{\ast}$};  \node at (0.55,-0.2) {\scriptsize  $g_m^{\ast}$};  \node at (-0.6,-0.2) {\tiny $\cdot\!\!\cdot\!\!\cdot$};  \node at (0.2,-0.2)  {\tiny $\cdot\!\!\cdot\!\!\cdot$}; 
\node at (-0.8,-0.22) {\scriptsize  $k_1^{\ast}$}; \node at (-0.35,-0.22) {\scriptsize  $k_q^{\ast}$}; 
\node at (-0.6,0.15) {\tiny $\cdot\!\!\cdot\!\!\cdot$};\node at (0.2,0.15) {\tiny $\cdot\!\!\cdot\!\!\cdot$};
\node at (0.6,1.85) {\tiny $\cdot\!\!\cdot\!\!\cdot$};\node at (-0.2,1.85) {\tiny $\cdot\!\!\cdot\!\!\cdot$};
\node at (0,1) {\tiny $\cdot\!\!\cdot\!\!\cdot$};
\end{tikzpicture}
\enspace \raisebox{1.3cm}{$=$}\enspace  \begin{tikzpicture}
\draw[draw=none,fill=EggShell!60](-1,2)--(1,2)-- (1,0)-- (-1,0)--(-1,2);
\node(a)[rectangle,draw=black,inner sep=0.5mm,fill=white,rounded corners=0.5mm]  at (0,1) {\footnotesize ${\chi}_{({\bf h}\cdot{\bf f},{\bf g}\cdot{\bf k})}$};
\draw[thick] (-0.75,2)to[out=-90,in=120](a)to[out=100,in=-90](-0.2,2); \draw[thick] (-0.75,0)to[out=90,in=-120](a)to[out=-100,in=90](-0.2,0); \draw[thick] (0.75,2)to[out=-90,in=60](a)to[out=-60,in=90](0.75,0); \draw[thick] (0.2,2)to[out=-90,in=80](a)to[out=-80,in=90](0.2,0); 
\node at (-0.8,2.2) {\scriptsize $f_1^{\ast}$}; \node at (-0.2,2.2) {\scriptsize $f_n^{\ast}$};  \node at (-0.5,2.2) {\tiny $\cdot\!\!\cdot\!\!\cdot$}; \node at (0.82,2.2) {\scriptsize $h_p^{\ast}$}; \node at (0.2,2.2) {\scriptsize $h_1^{\ast}$};  \node at (0.5,2.2) {\tiny $\cdot\!\!\cdot\!\!\cdot$}; 
\node at (0.2,-0.2) {\scriptsize $g_1^{\ast}$};  \node at (0.85,-0.2) {\scriptsize  $g_m^{\ast}$};  \node at (-0.5,-0.2) {\tiny $\cdot\!\!\cdot\!\!\cdot$};  \node at (0.5,-0.2) {\tiny $\cdot\!\!\cdot\!\!\cdot$}; 
\node at (-0.8,-0.22) {\scriptsize $k_1^{\ast}$}; \node at (-0.18,-0.22) {\scriptsize $k_q^{\ast}$}; 
\node at (-0.475,0.15) {\tiny $\cdot\!\!\cdot\!\!\cdot$};\node at (0.475,0.15) {\tiny $\cdot\!\!\cdot\!\!\cdot$};
\node at (0.475,1.85) {\tiny $\cdot\!\!\cdot\!\!\cdot$};\node at (-0.475,1.85) {\tiny $\cdot\!\!\cdot\!\!\cdot$};
\end{tikzpicture} \enspace \raisebox{1.3cm}{$=$}\enspace  \begin{tikzpicture}
\draw[draw=none,fill=EggShell!60](-1,2)--(1,2)-- (1,0)-- (-1,0)--(-1,2);
\node(a)[rectangle,draw=black,inner sep=0.5mm,fill=white,rounded corners=0.5mm]  at (0.2,1.35) {\footnotesize ${\chi}_{({\bf h},{\bf g}\cdot{\bf l})}$};
\node(b)[rectangle,draw=black,inner sep=0.5mm,fill=white,rounded corners=0.5mm]  at (-0.2,0.65) {\footnotesize ${\chi}_{({\bf l}\cdot {\bf f},{\bf k})}$};
\draw[thick] (-0.1,2)to[out=-90,in=120](a)to[out=60,in=-90](0.5,2); 
 \draw[thick] (-0.5,0)to[out=90,in=-120](b)to[out=-60,in=90](0.1,0);  
\draw[thick] (a) to[out=-90,in=40] (b) to[out=90,in=-140] (a);
\draw[thick] (0.75,0) to[out=90,in=-40] (a) to[out=-60,in=90](0.4,0);  
\draw[thick] (-0.75,2) to[out=-90,in=140] (b) to[out=120,in=-90](-0.4,2);
\node at (-0.8,2.2) {\scriptsize  $f_1^{\ast}$}; \node at (-0.35,2.2) {\scriptsize $f_n^{\ast}$};  \node at (-0.6,2.2)  {\tiny $\cdot\!\!\cdot\!\!\cdot$}; \node at (0.55,2.18) {\scriptsize   $h_p^{\ast}$}; \node at (-0.05,2.2) {\scriptsize $h_1^{\ast}$};  \node at (0.2,2.2)  {\tiny $\cdot\!\!\cdot\!\!\cdot$}; 
\node at (0.35,-0.2){\scriptsize   $g_1^{\ast}$};  \node at (0.82,-0.2) {\scriptsize  $g_m^{\ast}$};  \node at (-0.25,-0.2) {\tiny $\cdot\!\!\cdot\!\!\cdot$};  \node at (0.55,-0.2)  {\tiny $\cdot\!\!\cdot\!\!\cdot$}; 
\node at (-0.5,-0.22) {\scriptsize  $k_1^{\ast}$}; \node at (0.05,-0.22) {\scriptsize  $k_q^{\ast}$}; 
\node at (0.6,0.15) {\tiny $\cdot\!\!\cdot\!\!\cdot$};\node at (-0.2,0.15) {\tiny $\cdot\!\!\cdot\!\!\cdot$};
\node at (-0.6,1.85) {\tiny $\cdot\!\!\cdot\!\!\cdot$};\node at (0.2,1.85) {\tiny $\cdot\!\!\cdot\!\!\cdot$};
\node at (0,1) {\tiny $\cdot\!\!\cdot\!\!\cdot$};
\end{tikzpicture}}\end{center}
\end{itemize}
\end{lemma}
\begin{proof}
$\chi$-invertibility follows by vertical pasting and identity, generalized whiskering follows by iterated left and right whiskering and horizontal pasting follows by generalized whiskering and vertical pasting.  
\end{proof}

\smallskip

We now prove the key coherence result. 
 \begin{theorem}[$\chi$-coherence]\label{chi-coherence}
The 2-category $\mathsf{Fib}_{\cal C}$ is 2-thin. 
\end{theorem}

\begin{proof} 
The theorem is a  consequence of the   {\em convergence property} {\sf{(CP)}} of the rewriting system $(\mathsf{Fib}_{\cal C},\leadsto)$ defined on string diagrams of $\mathsf{Fib}_{\cal C}$ by orienting the relations (1), (2), (3.1) and (3.2) from Definition \ref{def:chi-signature} from left to right:
\begin{quote}
{\sf{(CP)}} Each string diagram $D$ has a unique   normal form ${\mathrm{NF}}(D)$.
\end{quote}  
We prove {\sf{(CP)}}. The normalization strategy for $D$ goes as follows. We first  factor $D$ as
$L\cdot \widehat D \cdot R$, where $\widehat D$ has no whiskers. Suppose that  $\widehat D:{\bf f}^{\ast}\Rightarrow {\bf g}^{\ast}$, where ${\bf f}^{\ast}$ (resp. ${\bf g}^{\ast}$) is either $f_n^{\ast}\cdots f_1^{\ast}$ (resp. $g_m^{\ast}\cdots g_1^{\ast}$) or ${{\mathrm{Id}}_{{\cal E}_B}}$. We   use vertical and horizontal pasting to succesively contract the edges of $\widehat D$, i.e., to merge pairs of adjacent vertices into a single vertex. Each rewrite step strictly reduces the number of vertices, while fixing the boundary 1-cells, so the process terminates at a single node diagram $\chi_{({\bf f},{\bf g})}$, if ${\bf f}^{\ast}\neq {\bf g}^{\ast}$, or at ${\mathrm{Id}}_{{\bf f}^{\ast}}$ if ${\bf f}^{\ast}={\bf g}^{\ast}$.  We finish by reattaching the whiskers $L$ and $R$ and applying generalized whiskering (Lemma \ref{easy}(b)), now oriented from left to right,  to obtain $\mathrm{NF}(D)=\chi_{(L\cdot{\bf f}\cdot R,L\cdot {\bf g}\cdot R)}$ (or $\mathrm{NF}(D)=\mathrm{id}_{L\cdot {\bf f}^{\ast}\cdot R}$). Therefore,  $\mathrm{NF}(D)$  corresponds to the generating 2-cell of the   polygon type determined by the boundary 1-cells, and thus it is unique. \end{proof}

Theorem 1 says that any  two 2-cells in  $\mathsf{Fib}_{\cal C}$ with the same source and target are equal. The practicality of this result  is reflected in the fact that  it allows us to replace any   string diagram  built on the pseudofunctorial signature  by another such diagram having the same boundary. We shall use this extensively in  Section 4 when proving the B\' enabou-Roubaud theorem, as a valid string diagram transformation within a larger class of string diagrams that we define next.

\subsection{The  2-category  $\mathsf{BiFib}_{\cal C}$ encoding the full bifibrational structure} We now extend the 2-category $\mathsf{Fib}_{\cal C}$ by adding left adjoints to pullback functors, as well as the inverses of Beck-Chevalley transformations.

\begin{definition}[The full bifibrational string-diagrammatic category ${\mathsf{BiFib}}_{\cal C}$]\label{fsd} The full bifibrational  string-diagrammatic category ${\mathsf{BiFib}}_{\cal C}$ is the quotient 
$${\mathsf{BiFib}}_{\cal C}:=\mathcal{F}(\Sigma_{{\cal C}}) \big/ {\mathsf{Rel}},$$
where the generating 2-polygraph $\Sigma_{{\cal C}}$ is defined as follows:\\[-0.1cm]
\begin{itemize}
\item 0-cells: $\Sigma^0_{\cal C}:=(\Sigma^{\ast}_{\cal C})^0$;\\[-0.1cm]
\item 1-cells: $\Sigma^1_{\cal C}:=(\Sigma^{\ast}_{\cal C})^1\cup \Sigma^{!}_{\cal C}$;\\[-0.1cm]
\item 2-cells: $\Sigma^2_{\cal C}:=(\Sigma^{\ast}_{\cal C})^2\cup \Sigma^{\mathsf{Adj}}_{\cal C} \cup \Sigma^{\mathsf{\overline{BC}}}_{\cal C}$,\\[-0.1cm]
\end{itemize}
where 
$$\begin{array}{rcl}
\Sigma^{!}_{\cal C}&:=&\{f_!:{\cal E}_A\rightarrow {\cal E}_B\,\,|\,\, f:A\rightarrow B\in {\it Arr}({\cal C})\},\\[0.2cm]
\Sigma^{\mathsf{Adj}}_{\cal C}&:=&\{\eta_h: \epsilon_{{\cal E}_C} \Rightarrow h^{*}  h_{!},\, \,\varepsilon_h:h_{!}  h^{*}\Rightarrow \epsilon_{{\cal E}_D}\,\,|\,\, h:C\to D\in {\it Arr}({\cal C})\},\\
\Sigma^{\mathsf{\overline{BC}}}_{\cal C}&:=&\biggl\{\overline{\sf{BC}}_P:b^{\ast}d_{!}\Rightarrow c_!a^{\ast}\,\,|\,\, P=
\raisebox{-0.85cm}{\begin{tikzpicture}
\node(a) at (0,0) {$C$};\node(b) at (1,0) {$D$};\node(c) at (0,1) {$A$};\node(d) at (1,0.975) {$B$};
\draw[->](a)--(b);\draw[->] (c)--(0.8,1) ;\draw[->](1,0.8)--(b);\draw[->](c)--(a);
\node(a') at (0.5,1.2) {\small $a$};
\node(b') at (0.5,-0.2) {\small $b$};
\node(c') at (1.2,0.5) {\small $d$};
\node(d') at (-0.2,0.5) {\small $c$};
\end{tikzpicture}} \mbox{ is a pullback square in } {\cal C}\biggr\},
\end{array}$$
and where ${\mathsf{Rel}}$ is the congruence of 2-cells generated by the generators of  ${\mathsf{Rel}}^{\ast}$ from Definition \ref{fff} and the following four relations: 
\begin{center}
\begin{tikzpicture}
\draw[draw=none,fill=EggShell!60](-1,2)--(1,2)-- (1,0)-- (-1,0)--(-1,2);
\node(al) [circle,fill=white,draw=black,inner sep=0.3mm] at (-0.35,1.5) {\footnotesize $\eta_h$};
\node(an) [circle,fill=white,draw=black,inner sep=0.25mm] at (0.35,0.5) {\footnotesize $\varepsilon_h$};
\draw[thick] (-0.75,0) to[out=90,in=-120] (al);\draw[thick] (al) to[out=-60,in=120] (an);
\draw[thick] (0.75,2) to[out=-90,in=60] (an);
\node at (0.75,2.2) {\small $h_{!}$};\node at (-0.75,-0.2) {\small $h_!$};
\end{tikzpicture}\enspace \raisebox{1.3cm}{$\overset{(4.1)}{=}$}\enspace \begin{tikzpicture}
\draw[draw=none,fill=EggShell!60](-1,2)--(1,2)-- (1,0)-- (-1,0)--(-1,2);
\draw[thick] (0,0) to[out=90,in=-90] (0,2);
\node at (0,2.2) {\small $h_{!}$};\node at (0,-0.2) {\small $h_!$};
\end{tikzpicture} 
\quad\quad\quad\quad
\begin{tikzpicture}
\draw[draw=none,fill=EggShell!60](-1,2)--(1,2)-- (1,0)-- (-1,0)--(-1,2);
\node(al) [circle,fill=white,draw=black,inner sep=0.3mm] at (0.35,1.5) {\footnotesize $\eta_h$};
\node(an) [circle,fill=white,draw=black,inner sep=0.25mm] at (-0.35,0.5) {\footnotesize $\varepsilon_h$};
\draw[thick] (-0.75,2) to[out=-90,in=120] (an);\draw[thick] (al) to[out=-120,in=60] (an);
\draw[thick] (0.75,0) to[out=90,in=-60] (al);
\node at (-0.75,2.2) {\small $h^{\ast}$};\node at (0.75,-0.2) {\small $h^{\ast}$};
\end{tikzpicture}\enspace  \raisebox{1.3cm}{$\overset{(4.2)}{=}$} \enspace \begin{tikzpicture}
\draw[draw=none,fill=EggShell!60](-1,2)--(1,2)-- (1,0)-- (-1,0)--(-1,2);
\draw[thick] (0,0) to[out=90,in=-90] (0,2);
\node at (0,2.2) {\small $h^{\ast}$};\node at (0,-0.2) {\small $h^{\ast}$};
\end{tikzpicture}
\end{center}

\begin{center}
\begin{tikzpicture}
\draw[draw=none,fill=EggShell!60](-1,2)--(1,2)-- (1,0)-- (-1,0)--(-1,2);
\node(a)[rectangle,draw=black,inner sep=0.3mm,fill=white,rounded corners=0.5mm]  at (0,0.6) {\footnotesize $\overline{\sf{BC}}_P$};
\node(b)[rectangle,draw=black,inner sep=0.5mm,fill=white,rounded corners=0.5mm]  at (0,1.4) {\footnotesize $\sf{BC}_P$};
\draw[thick] (-0.5,2)to[out=-90,in=120](b)to[out=-120,in=120](a) to[out=60,in=-60] (b) to[out=60,in=-90] (0.5,2);
\draw[thick] (-0.5,0) to[out=90,in=-120] (a)to[out=-60,in=90](0.5,0);
\node at (-0.475,2.2) {\small $a\!^{\ast}$};\node at (0.5,2.15) {\small $c_!$};
\node at (-0.475,-0.15) {\small $a\!^{\ast}$};\node at (0.5,-0.2) {\small $c_!$};
\end{tikzpicture}\enspace \raisebox{1.3cm}{$\overset{(5.1)}{=}$}\enspace \begin{tikzpicture}
\draw[draw=none,fill=EggShell!60](-1,2)--(1,2)-- (1,0)-- (-1,0)--(-1,2);
\draw[thick] (-0.5,0) to[out=90,in=-90] (-0.5,2);\draw[thick] (0.5,0) to[out=90,in=-90] (0.5,2);
\node at (-0.475,2.2) {\small $a\!^{\ast}$};\node at (0.5,2.15) {\small $c_!$};
\node at (-0.475,-0.15) {\small $a\!^{\ast}$};\node at (0.5,-0.2) {\small $c_!$};
\end{tikzpicture} \quad\quad\quad\quad \begin{tikzpicture}
\draw[draw=none,fill=EggShell!60](-1,2)--(1,2)-- (1,0)-- (-1,0)--(-1,2);
\node(a)[rectangle,draw=black,inner sep=0.5mm,fill=white,rounded corners=0.5mm]  at (0,0.6) {\footnotesize ${\sf{BC}}_P$};
\node(b)[rectangle,draw=black,inner sep=0.3mm,fill=white,rounded corners=0.5mm]  at (0,1.4) {\footnotesize $\overline{\sf{BC}}_P$};
\draw[thick] (-0.5,2)to[out=-90,in=120](b)to[out=-120,in=120](a) to[out=60,in=-60] (b) to[out=60,in=-90] (0.5,2);
\draw[thick] (-0.5,0) to[out=90,in=-120] (a)to[out=-60,in=90](0.5,0);
\node at (-0.5,2.2) {\small $d_!$};\node at (0.5,2.185) {\small $b^{\ast}$};
\node at (-0.5,-0.2) {\small $d_!$};\node at (0.5,-0.2) {\small $b^{\ast}$};
\end{tikzpicture}\enspace \raisebox{1.3cm}{$\overset{(5.2)}{=}$}\enspace \begin{tikzpicture}
\draw[draw=none,fill=EggShell!60](-1,2)--(1,2)-- (1,0)-- (-1,0)--(-1,2);
\draw[thick] (-0.5,0) to[out=90,in=-90] (-0.5,2);\draw[thick] (0.5,0) to[out=90,in=-90] (0.5,2);
\node at (-0.5,2.2) {\small $d_!$};\node at (0.5,2.185) {\small $b^{\ast}$};
\node at (-0.5,-0.2) {\small $d_!$};\node at (0.5,-0.2) {\small $b^{\ast}$};
\end{tikzpicture} 
\end{center}
with the 2-cell ${\sf{BC}}_P: c_!  a^{\ast}\Rightarrow b^{\ast}  d_{!}$ being defined by
 \begin{center}
\begin{tikzpicture}
\draw[draw=none,fill=EggShell!60](-1,2)--(1,2)--(1,0)--(-1,0)--cycle;
\node(f) [rectangle,draw=black,inner sep=0.5mm,fill=white,rounded corners=0.5mm]  at (0,1) {\footnotesize ${\sf{BC}}_P$};
\draw[thick] (-0.5,2)to[out=-90,in=120] (f) to[out=60,in=-90] (0.5,2);\draw[thick] (-0.5,0)to[out=90,in=-120] (f) to[out=-60,in=90] (0.5,0);
\node at (-0.5,2.2) {\small $a\!^{\ast}$};\node at (0.5,2.13) {\small $c_!$};
\node at (-0.5,-0.2) {\small $d_!$};\node at (0.5,-0.19) {\small $b^{\ast}$};
\end{tikzpicture} \enspace \raisebox{1.3cm}{$:=$} \enspace 
\begin{tikzpicture}
\draw[draw=none,fill=EggShell!60](-1,2)--(1,2)--(1,0)--(-1,0)--cycle;
\node(f) [rectangle,draw=black,inner sep=0.5mm,fill=white,rounded corners=0.5mm]  at (0,1) {\footnotesize $\chi_P$};
\node(al) [circle,fill=white,draw=black,inner sep=0.2mm] at (-0.35,1.5) {\footnotesize $\eta_d$};
\node(an) [circle,fill=white,draw=black,inner sep=0.3mm] at (0.35,0.5) {\footnotesize $\varepsilon_c$};
\draw[thick] (0.75,2) to[out=-90,in=60] (an);\draw[thick] (an) to[out=120,in=-60] (f)to[out=120,in=-60](al);
\draw[thick] (-0.75,0) to[out=90,in=-120] (al);
\draw[thick] (f)to[out=60,in=-90](0.3,2);
\draw[thick] (f)to[out=-120,in=90](-0.3,0);
\node at (0.3,2.2) {\small $a\!^{\ast}$};\node at (0.75,2.13) {\small $c_	!$};
\node at (-0.75,-0.2) {\small $d_!$};\node at (-0.28,-0.19) {\small $b^{\ast}$};
\end{tikzpicture} \raisebox{1cm}{.}
\end{center} \vspace{-0.2cm}\hfill$\square$
 \end{definition}

\begin{remark}[Mate construction and Beck--Chevalley transformation in ${\mathsf{BiFib}}_{\cal C}$]\label{mate}
Given a 2-cell \[
\alpha:\ g^{*}  f^{*}\Rightarrow k^{*} h^{*},
\] in ${\mathsf{BiFib}}_{\cal C}$, its {\it mate} is the composite
\[
\operatorname{mate}(\alpha)
\ :=\
(\varepsilon_k\cdot h^{*} f_{!})\ \circ\ ({k_{!}} \cdot \alpha \cdot  {f_{!}})\ \circ \ ({k_{!} g^{*}}\cdot  \eta_f  )
\ :\ k_{!}  g^{*}\Rightarrow h^{*}  f_{!},
\]
and its {\em two-fold mate} is the composite
 \[
\operatorname{2-mate}(\alpha)
\ :=\
(\varepsilon_h \cdot {f_! g_{!}})\ \circ\ ( {g_{!}} \cdot \operatorname{mate}(\alpha) \cdot  {h_{!}})\ \circ \ ({h_! k_{!}}\cdot \eta_g \cdot )
\ :\ h_!  k_{!}\Rightarrow  f_{!} g_!.
\]
The 2-cell ${\sf{BC}}_P: c_!  a^{\ast}\Rightarrow b^{\ast}  d_{!}$ from Definition \ref{fsd}   is the mate of the pseudofuncoriality 2-cell $\chi_P$ associated to the pullback square $P$ in ${\cal C}$. We shall refer to the composite 2-cell  ${\sf{BC}}_P$ as the 
 \emph{forward}  Beck--Chevalley transformation. In turn, the generator $\overline{\sf{BC}}_P$, which is the two-sided  inverse  of ${\sf{BC}}_P$ by (5.1) and (5.2), will  be called the \emph{backward} Beck--Chevalley transformation. Notice that the type of $\overline{\sf{BC}}_P$ does not match the type of $\operatorname{mate}( {\chi}_{P^{\it op}})$. This is the reason why $\overline{\sf{BC}}_P$ must appear as a primitive 2-cell in the definition of  ${\mathsf{BiFib}}_{\cal C}$.\hfill$\square$
\end{remark}

\begin{remark}\label{blahblah}
Although the string-diagrammatic 2-category ${\mathsf{BiFib}}_{\cal C}$ is presented relative to a fixed base category ${\cal C}$, this dependence is  only a matter of convenience, given that our intended interpretation of ${\mathsf{BiFib}}_{\cal C}$ is precisely a bifibration over ${\cal C}$. A fully general, base-independent 2-category  ${\mathsf{BiFib}}$ can be given by specifying the base pseudofunctorial signature relative to  purely syntactical  data given by ``object terms'', ``arrow terms'' and ``polygon-type terms'', and then extending it in the obvious way to the full bifibrational signature.  \hfill$\square$
\end{remark}

As indicated in Remark \ref{blahblah}, the string-diagrammatic 2-category ${\mathsf{BiFib}}_{\cal C}$ is constructed precisely in the way that recognizes Beck-Chevalley bifibrations over ${\cal C}$ as its models, in the  sense that each such    bifibration  $p:{\cal E}\rightarrow{\cal C}$  canonically determines a strict 2-functor
\[
\llbracket -\rrbracket_p: {\mathsf{BiFib}}_{\cal C} \to \mathsf{Cat}
\]
(and vice-versa). Indeed, $\llbracket -\rrbracket_p$ realizes\\[-0.1cm]
\begin{itemize}
\item the 0-cells of ${\mathsf{BiFib}}_{\cal C}$ as actual $p$-fibers over the objects of ${\cal C}$,\\[-0.1cm]
\item the  1-cells of ${\mathsf{BiFib}}_{\cal C}$ as (compositions of) actual pullback and direct image functors induced by cartesian and cocartesian lifts of $p$,  \\[-0.1cm]
\item  the  2-cells of ${\mathsf{BiFib}}^{\ast}_{\cal C}$ as  witnesses of the pseudofunctoriality of $f\mapsto f^{\ast}$; in particular the 2-cells  $\chi_{\emptyset_B^{b}}:({\mathrm{id}}_B)^{\ast}\Rightarrow \epsilon_B$  and $\chi_{({\bf f},{\bf g})}$, for ${\bf f}=C\xrightarrow{f\circ g} A$ and ${\bf g}=C\xrightarrow{g}B\xrightarrow{f}A$, are interpreted as canonical  natural isomorphisms $\iota_B$ and $\kappa_{f,g}$ from \eqref{canonical_isos_pseudo}, respectively,
 \\[-0.1cm]
\item the 2-cells of $\Sigma^{\mathsf{Adj}}_{\cal C}$ as units and  counits of adjunctions $f^{\ast}\dashv f_{!}:{{\cal E}_B}\rightarrow {{\cal E}_A}$, and\\[-0.1cm]
\item the 2-cells of $\Sigma^{\overline{\mathsf{BC}}}_{\cal C}$ as inverses of Beck-Chevalley transformations.\\[-0.1cm]
\end{itemize}
This allows us to reason about $p:{\cal E}\rightarrow{\cal C}$ by transfering along $\llbracket -\rrbracket_p$ the equalities of string diagrams proven in  ${\mathsf{BiFib}}_{\cal C}$. This is precisely our approach in proving the B\' enabou-Roubaud  theorem in the following section. More precisely, as announced in the introduction, we shall make three extensions of  ${\mathsf{BiFib}}_{\cal C}$, which provide syntactical frameworks for deriving  equalities  related to three particular  ways of characterizing  descent data.

\section{The B\' enabou-Roubaud  theorem}\label{thethm} 
\noindent Throughout this section, let ${\cal C}$ be a category with pullbacks, $p:{\cal E}\rightarrow {\cal C}$   a bifibration satisfying Beck-Chevalley property, and $f:B\rightarrow A$ a morphism in ${\cal C}$. In the first part of this section we describe three ways to characterize  descent data for an object $X\in {\it Ob}({\cal E}_B)$ and we state the equivalence of the three corresponding categories. The proof of the equivalence is given in the second part.

\subsection{The three categories of descent data} We recall the definitions of the Eilenberg–Moore category ${\mathsf{EM}}_p(T_f)$ and the classical category of descent data ${\mathsf{Desc}}_{p}(f)$, and we introduce the  category ${\mathsf{Act}}_p^{\mathsf{Eq}(f)}$, equivalent to the category of actions of the internal category ${\mathsf{Eq}(f)}$ on $p$ of Janelidze-Tholen from \cite{th}.

\medskip

\subsubsection{\bf The Eilenberg–Moore category ${\mathsf{EM}}_p(T_f)$}\label{algebra} Let
$T_f:=(f^{\ast}  f_{!},\eta_f,\mu)$, where $\mu:f^{\ast}\cdot\varepsilon_f\cdot f_!:T_f^{2}\rightarrow T_f$, be the monad on ${\cal E}_B$ coming from the adjunction $f^{\ast}\dashv f_{!}:{{\cal E}_B}\rightarrow {{\cal E}_A}$.  The  Eilenberg–Moore category ${\mathsf{EM}}_p(T_f)$ is the category whose objects are $T_f$-algebras, that is, pairs $(X,\alpha)$ of an object $X\in {\it{Ob}}({\cal E}_B)$ and a morphism $\alpha:T_f(X)\rightarrow X$   in ${\cal E}_B$ satisfying the following two axioms:\\[0.15cm]
\indent \hypertarget{(TA1)}{{\sf{(TA1)}}} $\alpha\circ ({\eta_f})_{X} ={\text{id}}_X$,\\[0.15cm]
\indent \hypertarget{(TA2)}{{\sf{(TA2)}}}  $\alpha\circ \mu_X=\alpha\circ T_f(\alpha)$. \\[0.15cm]
The morphisms $\phi:(X,\alpha)\rightarrow (Y,\alpha')$ are given by  arrows $\phi:X\rightarrow Y$ in ${\cal E}_B$ such that $\alpha\circ\phi=\alpha'\circ T(\phi)$. 

\smallskip

The   comparison functor $\Phi^{T_f}_p:{\cal E}_A\rightarrow{\mathsf{EM}}_{p}(T_f)$ from the Eilenberg-Moore factorization  \eqref{emfact} is defined by
$$Y\mapsto (f^{\ast}Y, f^{\ast}\cdot(\varepsilon_f)_Y).$$
The functor $f^{\ast}:{\cal E}_A\rightarrow {\cal E}_B$ is said to be  {\em monadic} if $\Phi^{T_f}_p$ is  an equivalence of categories. 

\medskip

\subsubsection{\bf The category of descent data ${\mathsf{Desc}}_{p}(f)$}  Let $f_1,f_2:B\times_A B\rightarrow B$ be the kernel pair of  $f:B\rightarrow A$, and $\pi_1,\pi_2:(B\times_A B)\times_B (B\times_A B)\rightarrow B\times_A B$  the pullback projections for $f_1$ and $f_2$, such that $f_2\circ\pi_2=f_1\circ\pi_1$.  Denote by $\pi=\langle f_1\pi_2,f_2\pi_1\rangle:(B\times_A B)\times_B (B\times_A B)\rightarrow B\times_A B$ the unique morphism whose existence is guaranteed by the universality of the pullback $(B\times_A B,f_1,f_2)$ of $f$ along itself, with respect to the triple $((B\times_A B)\times_B (B\times_A B),f_1\pi_2,f_2\pi_1)$. Finally,  let $\Delta=\langle{\it id}_B,{\it id}_B\rangle:B\rightarrow B\times_A B$ be the diagonal. 

\smallskip

In order to specify concisely the conditions that descent data must satisfy, we first introduce  abbreviations for certain compositions of pseudofunctoriality canonical isomorphisms \eqref{canonical_isos_pseudo}: we set 
\begin{equation}
d_i:=\kappa_{\Delta,f_i}^{-1}\circ\iota_B, \quad   j_i:=\kappa^{-1}_{f_i,\pi_{3-i}}\circ \kappa_{f_i,\pi},\enspace \mbox{ for } i=1,2,\quad \mbox{ and } \enspace j:=\kappa_{f_2,\pi_2}\circ\kappa^{-1}_{f_1,\pi_1}.
\label{someisos}
\end{equation}

\smallskip

An isomorphism $\varphi:f_1^{\ast}X\rightarrow f_2^{\ast}X$ in ${\cal E}_{B\times_A B}$ is called {\em a descent datum  for} $X\in{\it Ob}({\cal E}_B)$ if the following two conditions are satisfied:\\[0.15cm]
\indent \hypertarget{(DD1)}{{\sf{(DD1)}}}  $(d_2^{-1})_X\circ \Delta^{\ast}\varphi\circ (d_1)_X={\text{id}}_X$,\\[0.15cm]
\indent \hypertarget{(DD2)}{{\sf{(DD2)}}} $(\pi_1^{\ast}\cdot\varphi)\circ j_X\circ (\pi_2^{\ast}\cdot\varphi)=(j_2)_X\circ (\pi^{\ast}\cdot\varphi) \circ (j^{-1}_1)_X$.\\[0.15cm]  
 \noindent The category ${\mathsf{Desc}}_{p}(f)$ of descent data  has as objects  pairs $(X,\varphi)$ of an object $X\in {\it{Ob}}({\cal E}_B)$   and a descent datum $\varphi:f_1^{\ast}X\rightarrow f_2^{\ast}X$  for $X$. A morphism $\theta:(X,\varphi)\rightarrow (Y,\psi)$ in ${\mathsf{Desc}}_{p}(f)$  is given by an arrow $\theta:X\rightarrow Y$ in ${\cal E}_B$ such that $f^{\ast}_2 \theta \circ \varphi=\psi\circ f_1^{\ast} \theta$. 

\smallskip

The   comparison functor $\Phi^f_p:{\cal E}_A\rightarrow{\mathsf{Desc}}_{p}(f)$ from the descent factorization  \eqref{descfact} is defined by
$$Y\mapsto (f^{\ast}Y, (\kappa^{-1}_{f,f_2})_Y\circ (\kappa_{f,f_1})_Y).$$
 A morphism $f:B\rightarrow A$ in ${\cal C}$ is said to be an {\em  effective   descent morphism} if $\Phi^f_p$ is an equivalence of categories. 

\medskip

The following two remarks are about two equivalent ways to describe the objects of ${\mathsf{Desc}}_{p}(f)$.

\begin{remark}[About \hyperlink{(DD1)}{{\sf{(DD1)}}}] The conditions that make an isomorphism  $\varphi:f_1^{\ast}X\rightarrow f_2^{\ast}X$ a  descent data for $X$ were explicitly given for the first time by Janelidze-Tholen in \cite[Section 3]{JT}. The two conditions that they give are precisely our condition \hyperlink{(D2)}{{\sf{(DD2)}}} and, in place of our \hyperlink{(DD1)}{{\sf{(DD1)}}}, the following condition:\\[0.1cm]
\indent \hypertarget{(DD1)_{JT}}{{{\sf{(DD1)}}$_{\it{JT}}$}}  \enspace $(\hat{f_2})_X\circ \varphi\circ \delta_{1,X}={\mathrm{id}}_X$, \\[0.1cm]
where   $\delta_{i,X}:X\rightarrow f_i^{\ast}X$, for $i=1,2$, is the unique $p$-lift of $\Delta$ such  that $\hat{(f_i)}_X\circ \delta_{i,X}={\mathrm{id}}_X$. The condition \hyperlink{(DD1)_{JT}}{{{\sf{(DD1)}}$_{\it{JT}}$}} is equivalent to \hyperlink{(DD1)}{{\sf{(DD1)}}}. Indeed,  one sees immediately that \hyperlink{(DD1)_{JT}}{{{\sf{(DD1)}}$_{\it{JT}}$}}  is equivalent to $\varphi\circ \delta_{1,X}=\delta_{2,X}$, which is, in turn, equivalent to $\varphi\circ \hat{\Delta}_{f_1^{\ast}X}\circ (d_1)_X=\hat{\Delta}_{f_2^{\ast}X}\circ (d_2)_X$. By   the definition of $\Delta^{\ast}\varphi$, we conclude that the last equality holds if and only if $\Delta^{\ast}\varphi =(d_2)_X\circ(d^{-1}_1)_X$, which is precisely  \hyperlink{(DD1)}{{\sf{(DD1)}}}. Our choice to present \hyperlink{(DD1)}{{\sf{(DD1)}}} in this particular way is made so that  \hyperlink{(DD1)}{{\sf{(DD1)}}} has a direct string diagram representation, which is not the case for \hyperlink{(DD1)_{JT}}{{{\sf{(DD1)}}$_{\it{JT}}$}}, since its formulation involves explicit cartesian lifts.\footnote{Given a bifibration $p:{\cal E}\rightarrow {\cal C}$, one can refine the string diagram calculus so that it directly  encodes the structure of the {\em entire} total category  ${\cal E}$, not merely of the fibers. This   requires introducing an extra layer of boxes or frames,  such  that the content of each box is a string diagram in the standard sense, as well as special vertices for cartesian lifts, which are, exceptionally, always drawn outside all  boxes.}
\end{remark}
\begin{remark}[Indexed-family formulation of descent data] If $X\in{\it Ob}({\cal E}_B)$ is such that $X=f^{\ast}Y$ for some $Y\in{\it Ob}({\cal E}_A)$, then an isomorphism $\varphi_{g,h}:g^{\ast} X\rightarrow h^{\ast} X$ exists given any commutative square $f\circ g=f\circ h$ in ${\cal C}$, not necessarily a pullback. A descent data for $X$ can be given as the family of all such isomorphisms, i.e., the family indexed by all these commutative squares, which is natural in $(g,h)$, and for which we require the  following conditions to be satisfied: \\[0.15cm]
\indent \hypertarget{(DD1)'}{{\sf{(DD1)'}}}   $\varphi_{{\text{id}}_B,{\text{id}}_B}={\text{id}}_X$,\\[0.15cm]
\indent \hypertarget{(DD2)'}{{\sf{(DD2)'}}} $\varphi_{k,l}\circ\varphi_{h,k}=\varphi_{h,l}$,\\[0.15cm]
where, in \hyperlink{(DD2)'}{{\sf{(DD2)'}}}, $h,k,l$ and $f$ figure in a commutative cube, whose three edges are given by $f$ pointing to the origin of the cube $A$. The {\em indexed-family formulation} of descent data is equivalent to the {\em single-arrow formulation} on which we base  our definition of  ${\mathsf{Desc}}_{p}(f)$. The single arrow $\varphi:f_1^{\ast} X\rightarrow f_2^{\ast}X$ determines the entire family $\varphi_{h,k}$ by pullback along the universal map $\langle h,k\rangle:C\rightarrow B{\times}_A B$. Then, if $h=k={\it id}_B$, we get that $\langle h,k\rangle=\Delta:B\rightarrow B\times_A B$, and the condition \hyperlink{(DD1)'}{{\sf{(DD1)'}}}  translates to \hyperlink{(DD1)}{{\sf{(DD1)}}}  directly. In the same fashion, if we have three arrows $h,k,l:C\rightarrow B$ with $f\circ h=f\circ k=f\circ l$, then   the   condition \hyperlink{(DD2)'}{{\sf{(DD2)'}}} rewrites as   \hyperlink{(DD2)}{{\sf{(DD2)}}}. Conversely, if we start with the indexed family $\varphi_{g,h}$, we recover the single arrow as the special case $\varphi:=\varphi_{f_1,f_2}$. Now \hyperlink{(DD1)}{{\sf{(DD1)}}} and \hyperlink{(DD2)}{{\sf{(DD2)}}} follow by \hyperlink{(DD1)'}{{\sf{(DD1)'}}}  and    the instance $\varphi_{f_1\circ\pi_1,f_2\circ \pi_1}\circ\varphi_{f_1\circ\pi_2,f_2\circ\pi_2}=\varphi_{f_1\circ\pi_2,f_2\circ\pi_1}$ of  \hyperlink{(DD2)'}{{\sf{(DD2)'}}}, respectively, thanks to the naturality of the indexing.
\end{remark}

\medskip
\subsubsection{\bf The category ${\mathsf{Act}}_p^{\mathsf{Eq}(f)}$ of actions of the internal category ${\mathsf{Eq}(f)}$ on $p$} \label{actions} 
Recall that the kernel pair $f_1, f_2: B \times_A B\rightarrow B$ of $f$ gives rise to an internal category ${\mathsf{Eq}(f)}$  in $\mathcal{C}$ as follows:\\[-0.15cm]
\begin{itemize}
  \item  the ``object of objects"  ${\mathsf{Eq}(f)}_0 = B$,\\[-0.15cm]
  \item the ``object of morphisms" ${\mathsf{Eq}(f)}_1 = B \times_A B$,\\[-0.12cm]
  \item the source morphism $f_1: B\times_A B  \rightarrow B$ and the target morphism $f_2:  B\times_A B\rightarrow B$,\\[-0.15cm]
  \item the identity-assigning morphism   $\Delta=\langle {\text{id}_B},\text{id}_B\rangle :B\rightarrow B\times_A B$, and\\[-0.15cm]
  \item  the composition morphism $\pi=\langle f_1\pi_2, f_2\pi_1 \rangle:( B\times_A B)\times_B (B\times_A B) \rightarrow B\times_A B$.\\[-0.15cm]
 \end{itemize}

\smallskip

 In addition to the isomorphisms $d_i:{\mathrm{Id}}_{{\cal E}_B}\Rightarrow \Delta^{\ast} f_i^{\ast}$ and $j_i:\pi^{\ast}f^{\ast}_i\Rightarrow \pi^{\ast}_{3-i}f^{\ast}_i$, $i=1,2$,  defined in \eqref{someisos},  the definition of an action of ${\mathsf{Eq}(f)}$ on $p$ that we give below will also use the following natural transformations. We write $k_i:(f_i)_!(\pi_{3-1})_!\Rightarrow (f_i)_!\pi_!$,  $i=1,2$,  for the two-fold mate of the isomorphism $j_i$ (see Remark \ref{mate}). We define the natural transformations $\eta':{\mathrm{Id}}_{{\cal E}_B}\Rightarrow (f_2)_!f_1^{\ast}$ and $\mu': ({f_2})_! f_1^{\ast} ({f_2})_!f_1^{\ast}\rightarrow  ({f_2})_! f_1^{\ast}$ by $\eta'=(d_2^{-1}\cdot (f_2)_! f_1^{\ast})\circ (\Delta^{\ast}\cdot \eta_{f_2}\cdot f^{\ast}_1)\circ d_1$ and $\mu'=(({f_2})_!\cdot\varepsilon_{\pi}\cdot f_1^{\ast})\circ( k_2\cdot j^{-1}_1)\circ ({f_2}_!\cdot \overline{\mathsf{BC}}_P\cdot f_1^{\ast})$, respectively. 

\smallskip
 
  An \emph{action  of ${\mathsf{Eq}(f)}$ on $p$} is a pair $(X,\beta)$ of an object $X \in {\it Ob}(\mathcal E_B)$  and a morphism $\beta:({f_2})_{!}f^{\ast}_1 X\rightarrow X$ in $\mathcal E_B$, such that the following two conditions are satisfied:\\[0.15cm]
 \indent \hypertarget{(AC1)}{{\sf{(AC1)}}}  $\beta \circ \eta'_X={\text{id}}_X$, \\[0.15cm]
 \indent \hypertarget{(AC2)}{{\sf{(AC2)}}} $\beta\circ \mu'_X=\beta\circ ({f_2})_!f_1^{\ast}\beta$, \\[0.15cm]
 A morphism of actions $(X, \beta)$ and $(X', \beta')$   is a map $h: X \to X'$ in $\mathcal E_B$ such that
$h \circ \beta = \beta' \circ (f_2)_! f_1^* h$. This defines the  category ${\mathsf{Act}}_p^{\mathsf{Eq}(f)}$ of actions of ${\mathsf{Eq}(f)}$ on $p$.
 
\begin{remark}
Our  characterization of  ${\mathsf{Act}}_p^{\mathsf{Eq}(f)}$ is different than, but equivalent to, the original one due to Janelidze-Tholen. Their description of ${\mathsf{Act}}_p^{\mathsf{Eq}(f)}$ is crafted so that the equivalence  ${\mathsf{Desc}}_{p}(f)\simeq {\mathsf{Act}}_p^{\mathsf{Eq}(f)}$ can be delivered practically immediately;  in particular, their objects $(X,\beta)$ are such that the morphism $\beta$ comes in the form  $\beta:f_1^{\ast}X\rightarrow f_2^{\ast}X$ and it behaves exacly as a descent datum for $X$. Our description of $(X,\beta)$, on the other hand, portrays  directly a generalization of Bunge's notion \cite{bu} of internal (covariant) presheaf on ${\mathsf{Eq}(f)}$, in the sense that, when the given bifibration $p:{\cal E}\rightarrow {\cal C}$ is the codomain fibration ${\mathrm{cod}}:{\mathsf{Arr}}({\cal C})\rightarrow {\cal C}$, then $(X,\beta)$ is   such a presheaf  {\em on the nose}. To the best of our knowledge, this particular description of descent data has not been explicitly present in the existing literature. 
\end{remark}

\begin{theorem}[B\' enabou-Roubaud, Janelidze-Tholen]\label{equiv}
 If  a bifibration $p:{\cal E}\rightarrow {\cal C}$ satisfies the Beck-Chevalley condition, then, for each morphism $f:B\rightarrow A$ in ${\cal C}$, ${\mathsf{EM}}(T_f)\simeq{\mathsf{Desc}}_{\cal C}(f)\simeq {\mathsf{Act}}_p^{\mathsf{Eq}(f)}$. 
 \end{theorem}

\subsection{The proof} Before embarking on the proof, for the sake of readability,  we first establish certain string diagrammatic conventions for ${\mathsf{BiFib}}_{\cal C}$ and its extensions that we shall introduce along the way. The correspondence between the colours of the 2-dimensional regions  and the 0-cells (encoding categories) relevant for the proof is the following:
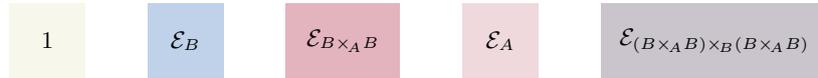
\begin{center}
\begin{tikzpicture}
\draw[draw=none,fill=EggShell!60](0,0)--(0,1)--(1,1)--(1,0)--cycle;
\node at (0.5,0.5) {\small ${1}$};
\end{tikzpicture} \quad\quad \begin{tikzpicture}
\draw[draw=none,fill=DustyRed!40](0,0)--(0,1)--(1,1)--(1,0)--cycle;
\node at (0.5,0.5) {\small ${\cal E}_B$};
\end{tikzpicture} \quad\quad \begin{tikzpicture}
\draw[draw=none,fill=SlateBlue!40](-0.25,0)--(-0.25,1)--(1.25,1)--(1.25,0)--cycle;
\node at (0.5,0.5) {\small ${\cal E}_{B\times\!_A B}$};
\end{tikzpicture} \quad\quad \begin{tikzpicture}
\draw[draw=none,fill=SageGreen!40](0,0)--(0,1)--(1,1)--(1,0)--cycle;
\node at (0.5,0.5) {\small ${\cal E}_{A}$};
\end{tikzpicture}  \quad\quad \begin{tikzpicture}
\draw[draw=none,fill=Graphite!40](-1,0)--(-1,1)--(2,1)--(2,0)--cycle;
\node at (0.5,0.5) {\small ${\cal E}_{(B\times\!_A B)\times\!_B(B\times\!_A B)}$};
\end{tikzpicture}
\end{center}
where $1$ is the symbol encoding the terminal category.
All vertices corresponding to pullback pseudofunctoriality isomorphisms   will be drawn as \raisebox{-0.05cm}{\begin{tikzpicture}
\node(c)[rectangle,draw=black,inner sep=0.5mm,fill=white,rounded corners=0.5mm]  at (0.25,0) {\footnotesize $\cong$};
\end{tikzpicture}}\,, and the ones corresponding to backward   Beck-Chevalley isomorphisms as  \raisebox{-0.05cm}{\begin{tikzpicture}\node (a) [rectangle,draw=black,inner sep=0.5mm,fill=white,rounded corners=0.5mm] at (0,0) {\scriptsize $\overline{\text{BC}}$};
\end{tikzpicture}}\,. The pseudofunctoriality isomorphisms related to the  direct image functor, as well as  forward Beck-Chevalley isomorphisms, will be drawn as \raisebox{-0.05cm}{\begin{tikzpicture}
\node(c)[rectangle,draw=black,inner sep=0.5mm,fill=DustyRose!40,rounded corners=0.5mm]  at (0.25,0) {\footnotesize $\cong$};
\end{tikzpicture}}\, and \raisebox{-0.05cm}{\begin{tikzpicture}\node (a) [rectangle,draw=black,inner sep=0.5mm,fill=DustyRose!40,rounded corners=0.5mm] at (0,0) {\scriptsize ${\text{BC}}$};
\end{tikzpicture}}\,, respectively; the green colour is meant to indicate that these 2-cells are not primitive and hence that, at this place, the diagram can be ``unfolded'' into a   larger one, consisting of primitive 2-cells. We take the green-colour convention also in more general contexts given by  the extensions of ${\mathsf{BiFib}}_{\cal C}$, where it will be used as the  colour of all non-primitive 2-cells. The types of all the 2-cells wil always be either readable  from the strings adjacent to  corresponding vertices, or traceable from other parts of the diagram. 

In a few places in the proof, we shall address Beck-Chevalley isomorphisms in the standard  form of a written formula;   the   pullback squares  $f\circ f_2=f\circ f_1$ and $f_2\circ\pi_2=f_1\circ\pi_1$  that index those isomorphisms will be denoted by $P_1$ and $P_2$, respectively.

 \medskip
The proof of Theorem \ref{equiv} proceeds by defining a triangle of functors
\begin{equation}
\raisebox{-1cm}{\begin{tikzpicture}
\node (a) at (0,0) {${\mathsf{EM}}_p(f)$};
\node (b) at (2,1) {${\mathsf{Desc}}_p(f)$};
\node (c) at (4,0) {${\mathsf{Act}}_p^{\mathsf{Eq}(f)}$};
\draw[->] (a) to(1.5,0.75); \draw[->]  (2.5,0.75) to(3.3,0.3); \draw[->]  (c)to(a);
\node(F) at (0.85,0.65) {\small $F$};\node(G) at (3.15,0.65) {\small $G$};
\node(H) at (2,-0.35) {\small $H$};
\end{tikzpicture}}
\label{tria}
\end{equation}
and showing that each cycle therein is the identity functor on the corresponding category. We do this through the string-diagrammatic environment given by three conservative  extensions,
$${\mathsf{BiFib}}_{\cal C}(f,X,\alpha), \quad {\mathsf{BiFib}}_{\cal C}(f,X,\varphi) \quad \mbox{and} \quad {\mathsf{BiFib}}_{\cal C}(f,X,\beta),$$ of ${\mathsf{BiFib}}_{\cal C}$, which allow to express and manipulate diagrammatically  descent data given by chosen objects $(X,\alpha)\in {\mathsf{EM}}_p(f)$, $(X,\varphi)\in {\mathsf{Desc}}_p(f)$ and $(X,\beta)\in {\mathsf{Act}}_p^{\mathsf{Eq}(f)}$, respectively. More precisely, and  with the above colour convention, the string-diagrammatic 2-categrory ${\mathsf{BiFib}}_{\cal C}(f,X,\alpha)$ (resp.  ${\mathsf{BiFib}}_{\cal C}(f,X,\varphi)$, ${\mathsf{BiFib}}_{\cal C}(f,X,\beta)$) is obtained by extending ${\mathsf{BiFib}}_{\cal C}$ with the  generating 2-cell 
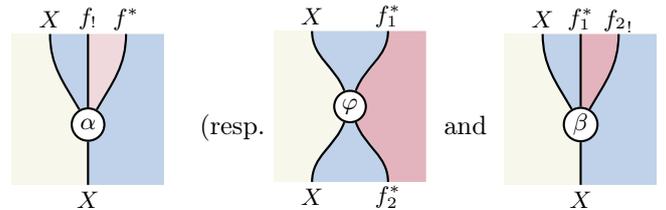
\begin{center}
{\begin{tikzpicture}
\draw[draw=none,fill=EggShell!60](-1,1)--(1,1)--(1,-1)--(-1,-1)--cycle;
\draw[draw=none,fill=DustyRed!40] (0,-0.2) to[out=90,in=-90] (0,1) -- (-0.5,1) to [out=-90,in=120] (0,-0.2);
\draw[draw=none,fill=DustyRed!40] (0,-0.2) to[out=60,in=-90] (0.5,1) -- (1,1)--(1,-1)--(0,-1) -- (0,-0.2); 
\draw[draw=none,fill=SageGreen!40] (0,-0.2) to[out=90,in=-90] (0,1) -- (0.5,1) to [out=-90,in=60] (0,-0.2); 
\node (a) [circle,draw=black,inner sep=0.65mm,thick,fill=white] at (0,-0.2) {\small $\alpha$};
\draw[thick] (a) to[out=120,in=-90] (-0.5,1); 
\draw[thick] (a) to[out=60,in=-90] (0.5,1); 
\draw[thick] (0,-1)-- (a) to[out=90,in=-90] (0,1); 
\node at (0.5,1.2) {\small $f^{\ast}$};\node at (0,1.2) {\small $f_{!}$};
\node at (-0.5,1.2) {\small $X$};\node at (0,-1.2) {\small $X$};
\end{tikzpicture}}  \quad \raisebox{1.1cm}{\mbox{(resp.}}  {\begin{tikzpicture}
\draw[draw=none,fill=EggShell!60](-1,1)--(1,1)--(1,-1)--(-1,-1)--cycle;
\draw[draw=none,fill=DustyRed!40] (0,0) to[out=100,in=-90,looseness=1] (-0.5,1) -- (0.5,1) to [out=-90,in=80,looseness=1] (0,0);
\draw[draw=none,fill=DustyRed!40] (0,0) to[out=-100,in=90] (-0.5,-1) -- (0.5,-1) to [out=90,in=-70] (0,0); 
\draw[draw=none,fill=SlateBlue!40] (0,0) to[out=80,in=-90] (0.5,1) -- (1,1) -- (1,-1) -- (0.5,-1) to [out=90,in=-80] (0,0); 
\node (a) [circle,draw=black,inner sep=0.45mm,thick,fill=white] at (0,0) {\small $\varphi$};
\draw[thick] (a) to[out=110,in=-90] (-0.5,1); 
\draw[thick] (a) to[out=70,in=270] (0.5,1); 
\draw[thick] (a) to[out=-110,in=90] (-0.5,-1); 
\draw[thick] (a) to[out=-70,in=-270] (0.5,-1); 
\node at (0.5,1.2) {\small $f_1^{\ast}$};\node at (0.5,-1.2) {\small $f_2^{\ast}$};
\node at (-0.5,1.2) {\small $X$};\node at (-0.5,-1.2) {\small $X$};
\end{tikzpicture}}   \raisebox{1.1cm}{\mbox{ and }}  {\begin{tikzpicture}
\draw[draw=none,fill=EggShell!60](-1,1)--(1,1)--(1,-1)--(-1,-1)--cycle;
\draw[draw=none,fill=DustyRed!40] (0,-0.2) to[out=90,in=-90] (0,1) -- (-0.5,1) to [out=-90,in=120] (0,-0.2);
\draw[draw=none,fill=DustyRed!40] (0,0) to[out=60,in=-90] (0.5,1) -- (1,1)--(1,-1)--(0,-1) -- (0,-0.2); 
\draw[draw=none,fill=SlateBlue!40] (0,0) to[out=90,in=-90] (0,1) -- (0.5,1) to [out=-90,in=60] (0,-0.2); 
\node (a) [circle,draw=black,inner sep=0.35mm,thick,fill=white] at (0,-0.2) {\small $\beta$};
\draw[thick] (a) to[out=120,in=-90] (-0.5,1); 
\draw[thick] (a) to[out=60,in=-90] (0.5,1); 
\draw[thick] (0,-1)-- (a) to[out=90,in=-90] (0,1); 
\node at (0.5,1.175) {\small ${f_2}_{!}$};\node at (0,1.2) {\small $f_1^{\ast}$};
\node at (-0.5,1.2) {\small $X$};\node at (0,-1.2) {\small $X$};
\end{tikzpicture}} \raisebox{1.1cm}{\mbox{)}}
\end{center}
  subject to \hyperlink{(TA1)}{{\sf{(TA1)}}} and \hyperlink{(TA2)}{{\sf{(TA2)}}} (resp. \hyperlink{(DD1)}{{\sf{(DD1)}}} and \hyperlink{(DD2)}{{\sf{(DD2)}}},   \hyperlink{(AC1)}{{\sf{(AC1)}}} and \hyperlink{(AC2)}{{\sf{(AC2)}}}). The purpose of the  three extensions is to provide a framework for the verification of the defining properties of functors forming \eqref{tria},  and of the equality of their composite around the triangle with the identity. This is justified as follows. For each symbol $\star\in\{\alpha, \varphi, \beta\}$, the extension $\mathsf{BiFib}_{\cal C}(f,X,\star)$   enjoys the universal property that strict 2-functors $\mathsf{BiFib}_{\cal C}(f,X,\star)\rightarrow {\mathsf{Cat}}$ are in bijection with strict 2-functors $\mathsf{BiFib}_{\cal C}\rightarrow {\mathsf{Cat}}$ equipped with a choice of concrete  instance of  descent data. Consequently, any equality of string diagrams involving the added generator that is derivable in $\mathsf{BiFib}_{\cal C}(f,X,\star)$ is valid for every such choice. 
 
\medskip
 
We now proceed to define the triangle of functors \eqref{tria}.

\medskip
 
The functor $F:{\mathsf{EM}}_p(f)\rightarrow {\mathsf{Desc}}_p(f)$ is defined by $F(X,\alpha):=(X,\varphi)$, where 
 \begin{center}
{
}
\end{center}

The defining properties of $F$ are verified in ${\mathsf{BiFib}}_{\cal C}(f,X,\alpha)$.  In Figure \ref{phiisiso}, we show that  $\varphi$ is an isomorphism, with  the inverse $\psi:f_2^{\ast}X\rightarrow f_1^{\ast}X$ defined by replacing, in the definition of $\varphi$, the isomorphism $\chi_{P_1}:f_1^{\ast}f^{\ast}\Rightarrow f_2^{\ast}f^{\ast}$  by its inverse $\chi_{P^{\it op}_1}:f_2^{\ast}f^{\ast}\Rightarrow f_1^{\ast}f^{\ast}$  .
 
\begin{figure}[H]
 \begin{center}
{%
}

 \end{center} \label{phiisiso}\caption{The proof of the equality $\psi\circ \varphi={\text{id}}_{f_1^{\ast}X}$.}
\end{figure}
The sequences of string diagram transformations proving that $\varphi$ satisfies  descent conditions \hyperlink{DD1}{\textsf {(DD1)}} and \hyperlink{DD2}{\textsf {(DD2)}} are given in  Figure \ref{fibercalc} and Figure \ref{dd2}, respectively. 
\begin{figure}[H]
{%
} 
\label{fibercalc}
\caption{Derivation of  {\textsf {(DD1)}}.}
\end{figure}

\begin{figure}[H]
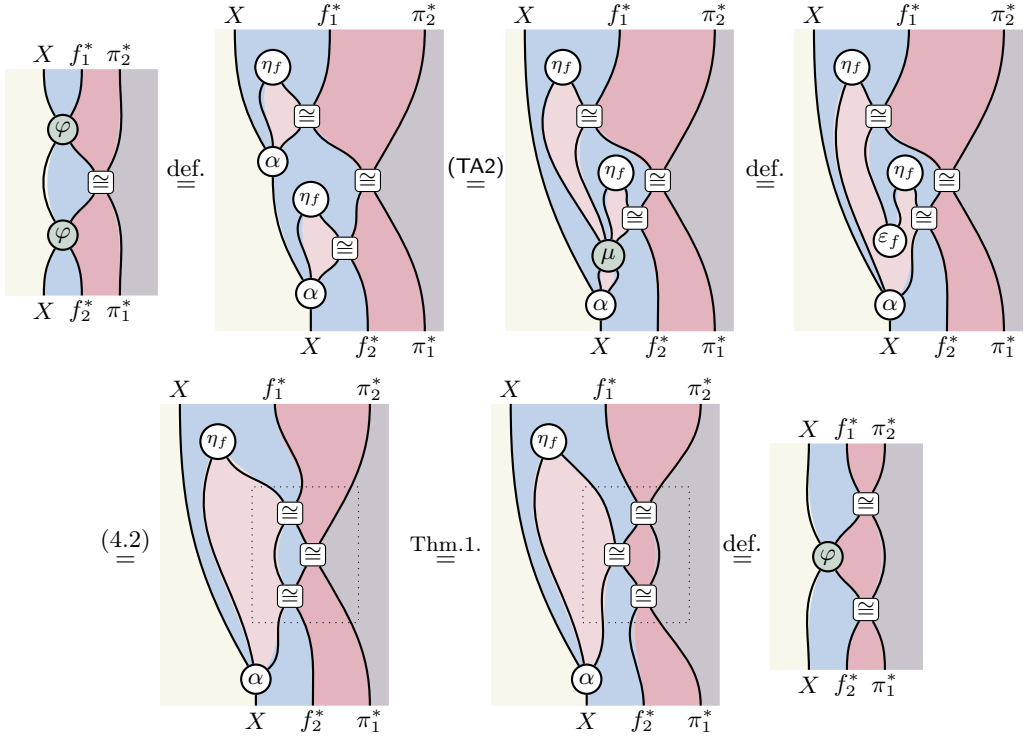

{%
} 
\caption{Derivation of {\sf{(DD2)}}.} \label{dd2}
\end{figure}

The functor $G:{\mathsf{Desc}}_p(f)\rightarrow {\mathsf{Act}}_p^{\mathsf{Eq}(f)}$ is defined by $G(X,\varphi):=(X,\beta)$, where 
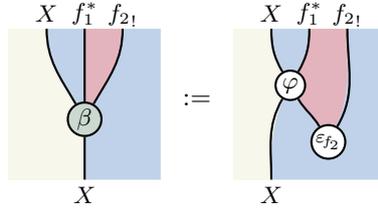
\begin{center}
{%
}
\end{center}
The verifications of the axioms  \hyperlink{(AC1)}{{\sf{(AC1)}}} and \hyperlink{(AC2)}{{\sf{(AC2)}}} for $\beta$, living  in ${\mathsf{BiFib}}_{\cal C}(f,X,\varphi)$, are given in Figure \ref{ac1} and  Figure \ref{ac2},  respectively. 

\begin{figure}[H]
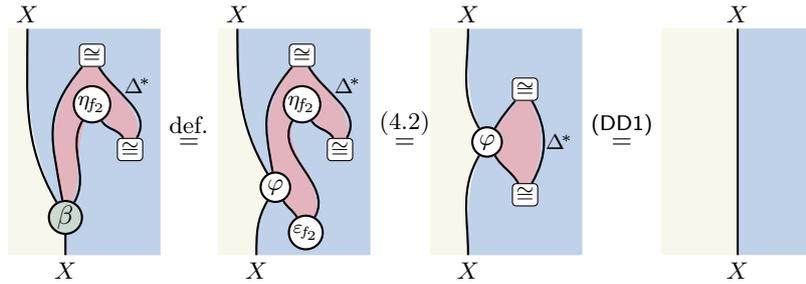

{%
} 
\label{ac1}
\caption{Derivation of {\sf{(AC1)}}.}
\end{figure}

\begin{figure}[H]
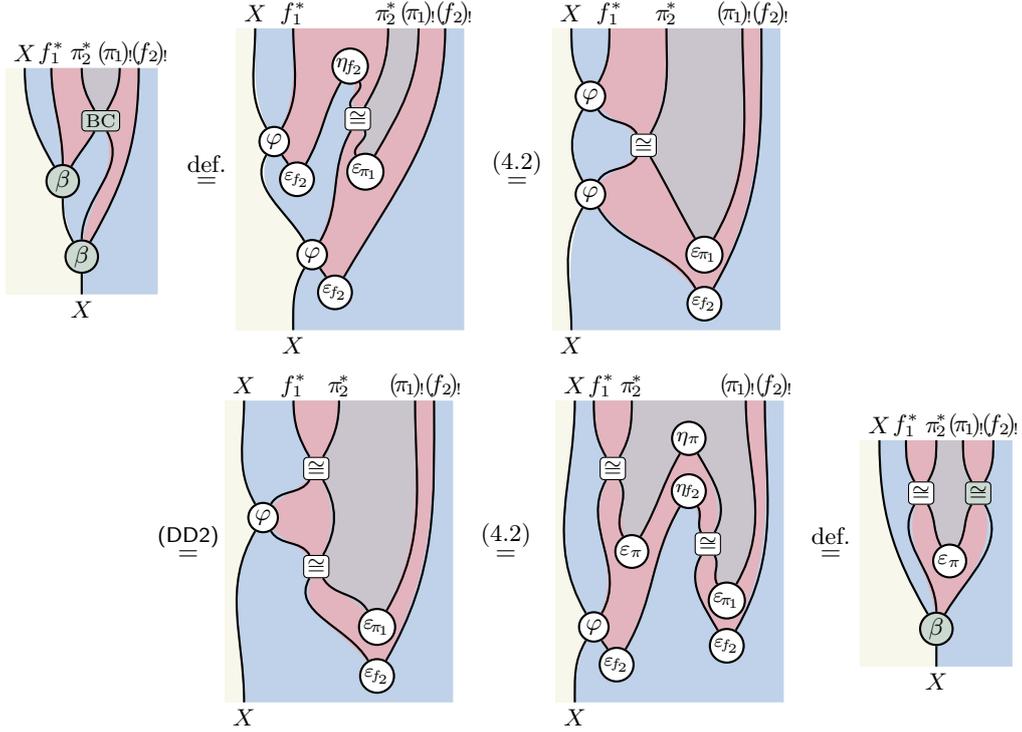

{%
} 
\label{ac2}
\caption{The axiom {\sf{(AC2)}} follows from the above derivation    by precomposing each string diagram in the sequence with the backward Beck-Chevalley isomorphism $\overline{\mathsf{BC}}_{P_2}:f_1^{\ast}(f_2)_!\Rightarrow (\pi_1)_!\pi_2^{\ast}$ and using the relation (5.2).}
\end{figure}

Finally, the functor $H:{\mathsf{Act}}_p^{\mathsf{Eq}(f)}\rightarrow {\mathsf{EM}}_p(f)$ is defined by $H(X,\beta):=(X,\alpha)$, where
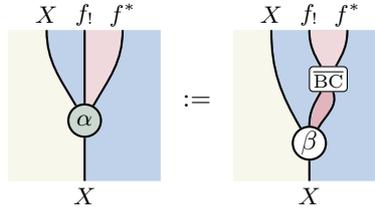
\begin{center}
{%
}
\end{center}
In order to prove that  $\alpha$ is a $T_f$-algebra, we first show that the natural transformations
$\eta':{\text{Id}}_{{\cal E}_B}\Rightarrow ({f_2})_!f_1^{\ast}$ and $\mu': ({f_2})_! f_1^{\ast} ({f_2})_!f_1^{\ast}\Rightarrow  ({f_2})_! f_1^{\ast}$,   defined in \S\ref{actions},   transfer along the Beck-Chevalley isomorphism ${\mathsf{BC}}_{P_1}:(f_2)_!f_1^{\ast}\Rightarrow  f^{\ast}f_!$ (and its inverse) precisely to $\eta_f: {\text{Id}}_{{\cal E}_B}\Rightarrow f^\ast f_!$ and $\mu_f: f^{\ast}f_!f^{\ast}f_!\Rightarrow f^{\ast}f_!$, defined in \S\ref{algebra}. More precisely, in Figure \ref{tran}  and Figure \ref{tram}, we verify the following two equalities:\\[0.15cm]
\indent  \hypertarget{($\eta'$-trans)}{{\sf ($\eta'$-trans)}}    ${\mathsf{BC}}_{P_1}\circ\eta'={\eta_f}$,\\[0.15cm]
\indent \hypertarget{($\mu'$-trans)}{{\sf ($\mu'$-trans)}}  $\mu'\circ (\overline{\mathsf{BC}}_{P_1}\cdot \overline{\mathsf{BC}}_{P_1})=\overline{\mathsf{BC}}_{P_1}\circ\mu$,\\[0.15cm]
  living in the base string-diagrammatic category ${\mathsf{BiFib}}_{\cal C}$.
  \begin{figure}[H]
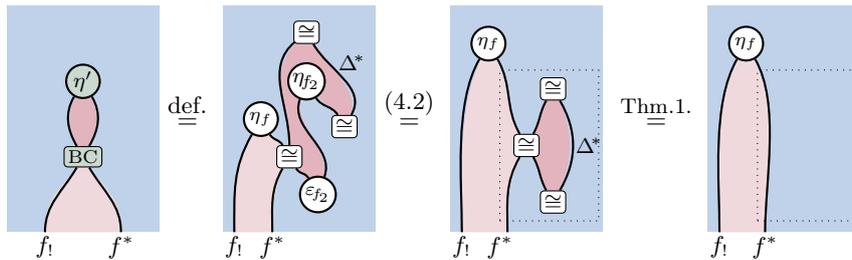

{
} 
\caption{Derivation of the equality {\sf ($\eta'$-trans)}.}
\label{tran}
\end{figure}

\begin{figure}[H]
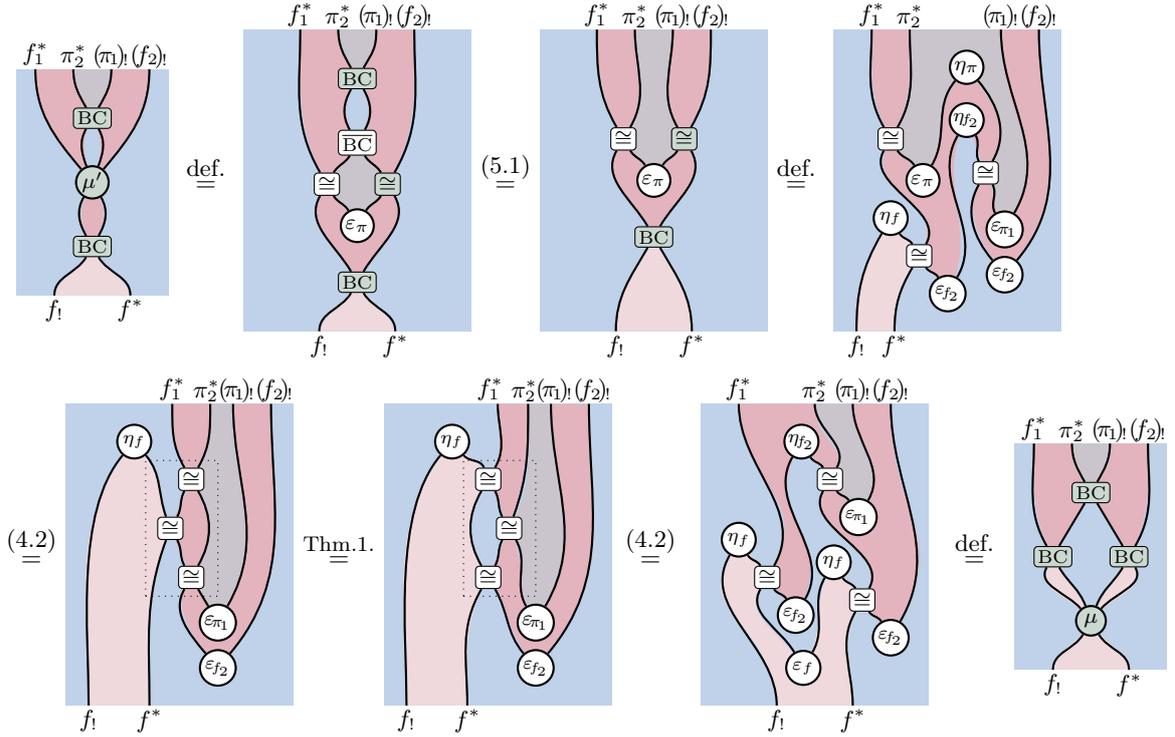

{%
} 
\caption{The equality {\sf ($\mu'$-trans)} is obtained by precomposing the string diagrams in the above derivation with  $\overline{\mathsf{BC}}_{P_2}\circ (\overline{\mathsf{BC}}_{P_1}\cdot \overline{\mathsf{BC}}_{P_1})$ and by postcomposing them with $\overline{\mathsf{BC}}_{P_1}$.}
\label{tram}
\end{figure}
The equalities \hyperlink{($\eta'$-trans)}{{\sf ($\eta'$-trans)}}  and \hyperlink{($\mu'$-trans)}{{\sf ($\mu'$-trans)}}  entail that  the endofunctor $({f_2})_!f_1^{\ast}:{\cal E}_B\rightarrow {\cal E}_B$, equipped with  $\eta'$ and $\mu'$, is a monad on ${\cal E}_B$ and, thanks to {\textsf {(AC1)}} and  {\textsf {(AC2)}},  that an object  $(X,\beta)$ of ${\mathsf{Act}}_p^{\mathsf{Eq}(f)}$ is   an algebra for that  monad.  Consequently,    precomposing $\beta$  with $\overline{\mathsf{BC}}_{P_1}$ indeed yields a genuine $T_f$-algebra structure. A direct verification of the algebra laws {\textsf {(TA1)}} and {\textsf {(TA2)}}, living in  ${\mathsf{BiFib}}_{\cal C}(f,X,\beta)$, is given in Figure \ref{algver}.
\begin{figure}[H]
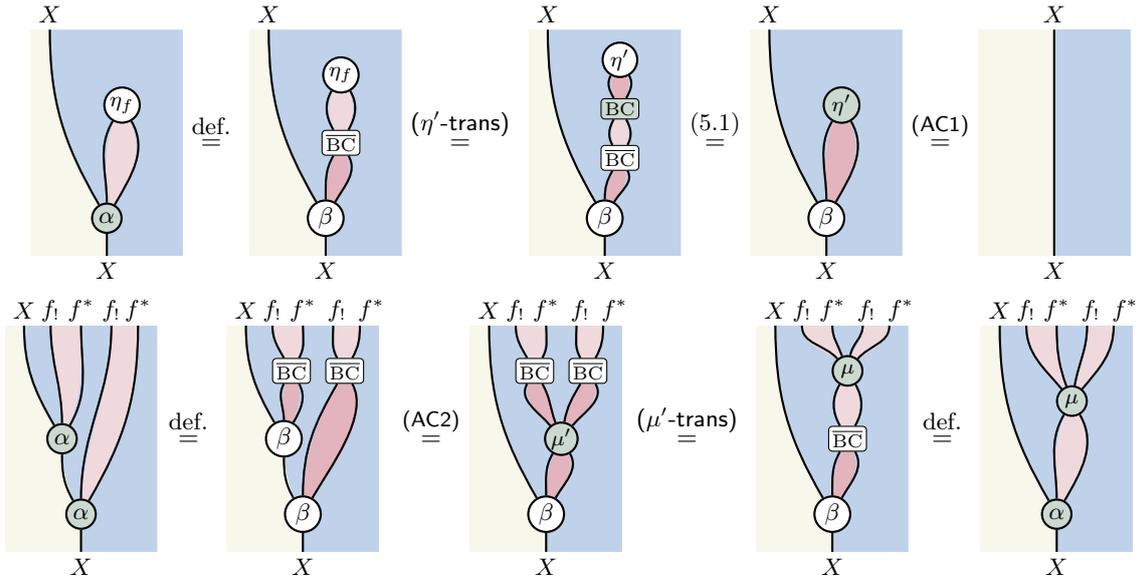

 {
}
\caption{Derivations of  {\textsf {(TA1)}} (top) and {\textsf {(TA2)}} (bottom).}
\label{algver}
\end{figure}
The last thing that we need to verify is that the endofunctors on the three categories  defined in the obvious way from $F$, $G$ and $H$ are the identities. We do this   in Figure \ref{fin}.
\begin{figure}[H]
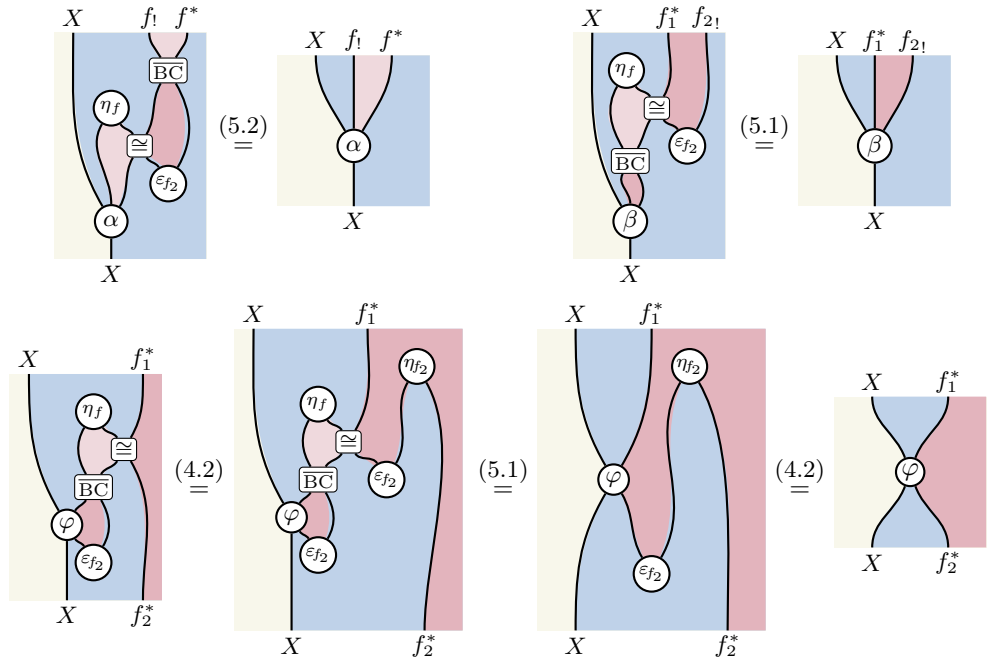

{%
} 
\caption{Derivations of equalities $H(G(F(X,\alpha)))=(X,\alpha)$ (top left), $G(F(H(X,\beta)))=(X,\beta)$ (top right), and $F(H(G(X,\varphi)))=(X,\varphi)$ (bottom).}
\label{fin}
\end{figure}
This finishes the proof of the B\' enabou-Roubaud theorem.
\bibliographystyle{unsrt}

\begin{thebibliography}{0000000}
\bibitem{br} J. B\' enabou and J. Roubaud, Monades et descente, {\it Comptes Rendus de
l'Acad\' emie des Sciences} 270, Serie A,  96-98 (1970).\\[-0.2cm]
\bibitem{bu} M. Bunge, Internal presheaves toposes, {\it Cahiers de topologie et géométrie différentielle catégoriques}, tome
18, no 3,   291-330 (1977).\\[-0.2cm]
\bibitem{bur} A. Burroni, Higher-dimensional word problems with applications to equational logic, {\it Theoretical Computer Science} 115, no. 1, 43–62  (1993).
\bibitem{pl0} P.-L. Curien, Substitution up to isomorphism, {\it Fundamenta Informaticae}
19, 51-85 (1993) (preliminary version appeared as LIENS Technical report
90-9 (1990)).\\[-0.2cm]
\bibitem{D} A.  Delpeuch,  Word Problems on String Diagrams. University of Oxford (2021).\\[-0.2cm]
\bibitem{gro} A. Grothendieck, Technique de descente et théorèmes d'existence en géométrie algébrique. I. Généralités. Descente par morphismes fidèlement plats,  {\it S\' eminaire Bourbaki}, Exp. No. 190, vol. 5, Soc. Math.
France, Paris,  299–327 (1995).\\[-0.2cm]
\bibitem{Ho} G. Hotz. Eine Algebraisierung des Syntheseproblems von Schaltkreisen I.
{\it Elektronische Informationsverarbeitung und Kybernetik} 1(3), 185–205 (1965).\\[-0.2cm]
\bibitem{JT} G. Janelidze and W. Tholen, Facets of descent, I, {\it Appl. categorical struct.} 2, 245–281 (1994).\\[-0.2cm]
 \bibitem{th} G. Janelidze and W. Tholen, Facets of Descent, II, {\it Appl. categorical struct.} 5, 229-248 (1997). \\[-0.2cm]
\bibitem{BJ} B. Jacobs, Categorical logic and type theory, {\it Studies in logic and the foundations of mathematics}, vol. 141. Elsevier   (1999).\\[-0.2cm]
\bibitem{NY} N. Johnson and D. Yau, 2-Dimensional Categories. \href{https://arxiv.org/abs/2002.06055}{arXiv:2002.06055}, (2020).\\[-0.2cm]
\bibitem{JS88} A. Joyal and R. Street. Planar diagrams and tensor algebra. Unpublished
manuscript, available from Ross Street’s website,     (1988).\\[-0.2cm]
\bibitem{JoyalStreet1991} A. Joyal and R. Street,  The Geometry of Tensor Calculus. I , {\it Advances in Mathematics 88},  55–112 (1991).\\[-0.2cm]
\bibitem{BK} B. Kahn, On the Bénabou-Roubaud theorem, \href{https://arxiv.org/pdf/2404.00868}{arXiv:2404.00868},  (2025).\\[-0.2cm]
\bibitem{ML} M. Lucas, A coherence theorem for pseudonatural transformations, {\it Journal of Pure and Applied Algebra}, Volume 221, Issue 5, 1146-1217,  (2017).\\[-0.2cm]
\bibitem{FLN} F. Lucatelli Nunes, Pseudo-Kan extensions and descent theory, {\it Theory Appl. Categ. 33}, Paper No. 15, 390–444 (2018). \\[-0.2cm]
\bibitem{LR} F. Loregian and E. Riehl, Categorical notions of fibration, {\it Expositiones Mathematicae} 38, 496-514 (2020).\\[-0.2cm]
\bibitem{dm} D. Mardsen, {Category Theory Using String Diagrams},  \href{https://arxiv.org/abs/1401.7220}{arXiv:1401.7220}, (2014).\\[-0.2cm]
\end{thebibliography}

\end{document}